\documentclass[draft,a4paper,reqno]{amsart}%
\usepackage{amssymb}
\usepackage{amsmath}
\usepackage{amsfonts}
\usepackage{bm}%
\providecommand{\U}[1]{\protect\rule{.1in}{.1in}}

\let\oldmathbf\mathbf
\renewcommand{\mathbf}[1]{\boldsymbol{\oldmathbf{#1}}}
\newtheorem{theorem}{Theorem}

\newtheorem{definition}[theorem]{Definition}

\newtheorem{lemma}[theorem]{Lemma}

\newtheorem{proposition}[theorem]{Proposition}

\allowdisplaybreaks
\begin{document}
\title{Equiconvergence for perturbed Jacobi polynomial expansions}
\author[G. Gigante]{Giacomo Gigante}
\address{Dipartimento di Ingegneria Gestionale, dell'Informazione e della Produzione,
Universit\`a degli Studi di Bergamo, Viale Marconi 5, Dalmine BG, Italy}
\email{giacomo.gigante@unibg.it}
\author[K. Jotsaroop]{Jotsaroop Kaur}
\address{Department of Mathematics,
	Indian Institute of Science Education and Research, Mohali India} 
\email{jotsaroop@iisermohali.ac.in}


\begin{abstract}
We show asymptotic expansions of the eigenfunctions of certain perturbations of the Jacobi operator in a bounded interval, deducing equiconvergence results between expansions with respect to the associated orthonormal basis and expansions with respect to the cosine basis. Several results for pointwise convergence then follow. 
        
\thanks{Giacomo Gigante has been supported by an Italian GNAMPA 2019 project.}

\end{abstract}
\maketitle

\section{Introduction}
The equiconvergence principle for expansions in eigenfunctions of Sturm-Liouville
operators goes back to the beginning of the 20th century, with the seminal paper of 
A. Haar \cite[page 355]{Haar}. It essentially says that the expansion of an integrable 
function in $[0,\pi/2]$ with respect to the eigenfunctions of the Sturm-Liouville operator in Liouville normal form
\[
-u''+q(t)u
\]
with homogeneous separated boundary conditions 
\[
u'(0)-hu(0)=u'(\pi/2)-Hu(\pi/2)=0
\] 
converges or diverges at some point if its expansion with respect to the cosine basis
$\{\cos(2nx)\}_{n=0}^{+\infty}$ converges or diverges at that point.  The function $q$ is required to be of bounded variation. J. L. Walsh \cite{Walsh} studied the case with boundary conditions $u(0)=u(\pi/2)=0$, realizing that the equiconvergence is uniform in the interval (although he was only able to show it for for square integral functions). 

Soon after Haar's work, J. Tamarkin \cite{Tamarkin} generalized these results to expansions in eigenfunctions of certain higher order differential problems now known as Birkhoff-regular problems
\cite{Birkhoff}, obtaining equiconvergence for integrable functions when the coefficients of the 
equation are continuous. M. H. Stone \cite{Stone}, studied the case of integrable coefficients.  

Several generalizations followed, some have been collected by A. M. Minkin in the survey \cite{Minkin}. Here we mention one of Minkin's own results \cite[Theorem 0.1, page 3673]{Minkin}. He studied symmetric quasi differential operators of higher order, defined on any open interval $I$, bounded or unbounded, with locally integrable coefficients, establishing an equiconvergence result for functions in $L^2(I)$, or integrable but compactly supported away from the endpoints of $I$.

Here we are interested in a singular Sturm-Liouville operator again of the form
\[
\ell u=-u''+q(t)u
\]
for particular choices of the function $q$, which will only be {\em locally integrable} in the interval $(0,\pi/2)$. Our 
goal is to obtain an equiconvergence 
result for all functions in $L^1((0,\pi/2))$. Unfortunately, to the best of our knowledge, none of the known 
results apply to our case, not even the above mentioned theorem of Minkin (as it does not apply to  functions
in $L^1$ of the {\em whole interval}). 
The reason why it is so important for us to show equiconvergence for such a large class of functions
will become clear by the end of this introduction. In fact, it would be interesting to obtain
results for even larger classes of functions. See also \cite{C} on these matters in a classic context.

The present paper is divided into three parts. In the first part, \S\ref{abstract}, we show that if an orthonormal basis $\{u_n\}_{n=0}^{+\infty}$ of
$L^2([0,\pi/2],dt)$, satisfies five axiomatic properties of resemblance
with the basis $\{\cos(2nt)\}_{n=0}^{+\infty}$, then there is equiconvergence between the Fourier
series associated respectively with the above two basis, with essentially any summability method.

More precisely, the summability method will come from a bounded double sequence $\{r_{n,N}\}$
such that $\sum_{n=0}^{+\infty}|r_{n,N}-r_{n+1,N}|$ is also bounded in $N$. For example, one could take
the Ces\`aro means of order $\theta\ge0$, given by
\[
r_{n,N}=\frac{A_{N-n}^\theta}{A_N^\theta}
\]
where $A_n^\theta = {{n+\theta}\choose n}$ for $n\ge0$, and $A_n^\theta =0$ for $n<0$. Thus, we shall call
\[
T_Nf(x):=\sum_{n=0}^{+\infty}r_{n,N}\int_0^{\pi/2}f(t)u_n(t)dt u_n(x)
\]
the means of the Fourier series of a function $f$, computed with our favourite summability method, with respect 
to the basis $\{u_n\}$, and
\[
D_Nf(x):=\frac 2\pi r_{0,N}\int_0^{\pi/2}f(t)dt+\frac 4\pi\sum_{n=1}^{+\infty}r_{n,N}\int_0^{\pi/2}f(t)\cos(2nt)dt \cos(2nx)
\]
the means of the Fourier series of $f$ with respect to the basis $\{\cos(2nt)\}$ with the same summability method.

In Theorem \ref{equiconv_L1} we show that if there is a dense subspace $\Omega$ of $L^1([0,\pi/2],dt)$ such that for all $t\in(0,\pi/2)$ and for all
$g\in\Omega$,
\[
\lim_{N\to+\infty} T_Ng(t)-D_Ng(t)=0,
\]
then for all $t\in(0,\pi/2)$ and for all
$f\in L^1([0,\pi/2],dt)$
\[
\lim_{N\to+\infty} T_Nf(t)-D_Nf(t)=0,
\]
and uniform convergence away from the endpoints in $\Omega$ implies the same type of convergence in $L^1$.
This result is very much in the spirit of J. E. Gilbert's work \cite{G1}. The five conditions on the basis $\{u_n\}$ and the hypothesis on $r_{n,N}$ are basically the same here and in \cite{G1}, but Gilbert proves a transplantation theorem, rather than an equiconvergence theorem. In a few words, he shows, among other results, that if the maximal operator $D^*f(x)=\sup_N|D_Nf(x)|$ is of weak type
$(p,p)$ with respect to $L^p((0,\pi/2),dt)$ for some $p$, then so is $T^*f(x)=\sup_N|T_Nf(x)|$. 

Observe that an equiconvergence result like the one we show here completely bypasses maximal functions, and can be used to transfer basically any type of pointwise convergence result  (convergence in a given point, divergence sets for continuous functions, almost everywhere convergence for $L^p$ functions, Hausdorff dimension of divergence sets of regular functions, etc.) that one has for, say, Ces\`aro means of classic Fourier series of functions in any space contained in $L^1([0,\pi/2],dt)$, to the same type of means of Fourier series associated with the basis 
$\{u_n\}$, for functions in the same space.

Since the boundedness of the maximal operator $T^*$ implies the almost everywhere convergence
of $T_Nf$, Gilbert's transplantation result can also be used to deduce almost everywhere results for $T_Nf$ from well known results on the boundedness of $D^*f$. 
Nevertheless, other types of pointwise convergence results do not seem to be deducible from Gilbert's transplantation theorem.

Despite this difference between the two theorems, Gilbert's and ours, the proof techniques are indeed very similar.

In the second part of the paper, \S\ref{L_normal}, we consider a perturbation of the Jacobi operator
\[
Ju=-u''+\left((\alpha^2-\frac14)\cot^2t+(\beta^2-\frac14)\tan^2t\right)u, \qquad \alpha\ge\beta>-1/2,
\]
that is an operator of the form
\[
\ell u=Ju-\chi(t) u,
\]
 where $\chi\in\mathcal C^2(\mathbb R)$ is even with respect to $0$ and $\pi/2$. The unbounded operator $\ell$ is defined on
 \[
 \mathcal D=\{f,f'\in AC_{\mathrm{loc}}, f,\ell f\in L^2([0,\pi/2],dt)\}
 \]
($AC_{\mathrm{loc}}$ is the class of absolutely continuous functions on all compact subintervals of $(0,\pi/2)$), and its
restriction $\tilde \ell$ to a properly chosen subdomain of $\mathcal D$ is self-adjoint and bounded below. By 
the spectral theorem there is a sequence of eigenfunctions of $\tilde\ell$, $\{u_n\}_{n=0}^{+\infty}$, that form an orthonormal basis 
of $L^2([0,\pi/2],dt)$.
The core of this second part is condensed in Theorem \ref{fits} and consists thus in showing that these eigenfunctions have certain asymptotic expansions and therefore satisfy the five axiomatic properties of the 
first part, so that the equiconvergence result holds in this context. As a consequence, virtually any pointwise convergence result for,
say, Ces\`aro means known for the cosine basis can be stated exactly in the same form for the basis $\{u_n\}$. The reader can find a short list 
at the end of \S\ref{L_normal}.

Once again, the proofs of some of the results here follow the lines of a paper by J. E. Gilbert \cite{G2}. Gilbert too shows that the eigenfunctions of certain operators satisfy his five axioms. Notice though that he requires that his operators satisfy some sort of symmetry with respect to the center of the interval $(0,\pi/2)$. In our case, this hypothesis would require $\alpha=\beta$. Thus, our result can be seen as an extension of Gilbert's result to more general operators.

In the third and final part of the paper, \S\ref{general}, we consider a perturbation $A(t)$ of the function  $\sin^{2\alpha+1}t\cos^{2\beta+1}t$ in the interval $(0,\pi/2)$, again with $\alpha\ge\beta>-1/2$. By this we mean that
\[
A(t)=B(t)\sin^{2\alpha+1}t\cos^{2\beta+1}t,
\]
where $B\in\mathcal C^4(\mathbb R)$ is strictly positive and even with respect to $0$ and $\pi/2$. The Sturm-Liouville
operator (now in general form)
\[
Lv:=\frac 1{A(t)}(A(t)v')'
\]
can be reduced, after the Liouville transformation $u(t)=A^{1/2}(t)v(t)$, to its Liouville normal form
\[
\ell u=Ju-\chi(t)u,
\]
where
\begin{align*}
\chi(t)&=(\beta+\frac12)\frac{B'(t)}{B(t)}\tan t-(\alpha+\frac12)\frac{B'(t)}{B(t)}\cot t+\frac 14\left(\frac{B'(t)}{B(t)}\right)^2-\frac 12\frac{B''(t)}{B(t)}\\
&+2\alpha\beta+2\alpha+2\beta+\frac 32
\end{align*}
is a $\mathcal C^2(\mathbb R)$ function even with respect to $0$ and $\pi/2$. Thus, the results of the second part of this paper can
be applied to this particular operator $\ell$, to deduce that the functions
\[
v_n(t):=A^{-1/2}(t)u_n(t), \qquad n=0,1,\ldots
\]
form an orthonormal basis of $L^2((0,\pi/2),A(t)dt)$ consisting of eigenfunctions of $L$. Various means of Fourier series
with respect to this type of bases have been studied widely in the past. Here we just want to mention the monograph \cite{CM}
on weak type estimates and almost everywhere convergence for Ces\`aro means of Jacobi polynomial expansions, which in our notation corresponds to the case $B(t)\equiv1$. The case of the unbounded interval $[0,+\infty)$ also has been studied, see \cite{BG, CCTV,CGV}.

The means operator in this context has the form
\begin{align*}
T_N^Af(x)&:=\sum_{n=0}^{+\infty}r_{n,N}\int_0^{\pi/2}f(t)v_n(t)A(t)dt v_n(x)\\
&=\sum_{n=0}^{+\infty}r_{n,N}\int_0^{\pi/2}f(t)A^{1/2}(t)u_n(t)dt A^{-1/2}(x)u_n(x)\\
&=A^{-1/2}(x) T_N(A^{1/2}f)(x).
\end{align*}
This implies that the equiconvergence  results of the first two parts of the paper, between $D_N$ and $T_N$,
can be transferred to $A^{-1/2}(x) D_N(A^{1/2}f)(x)$ and $T_N^Af(x)$, {\em as long as the function $A^{1/2}f$ belongs to
$L^1((0,\pi/2),dt)$}.

Concerning the almost everywhere convergence, looking at \cite[Theorem 1.4]{CM} and by analogy with what happens in the
case of the unbounded interval studied in \cite{CGV}, one would expect to obtain almost everywhere convergence for Ces\`aro means of order
$\theta>0$ for all functions in $L^p((0,\pi/2),A(t)dt)$ for 
\[
p\ge\max\left(1,\frac{4\alpha+4}{2\alpha+3+2\theta}\right).
\]
Unfortunately, observe that in order to have $A^{1/2}f\in L^1((0,\pi/2),dt)$ as needed to
apply equiconvergence, one needs $f\in L^p((0,\pi/2),A(t)dt)$ with
\[
p>\frac{4\alpha+4}{2\alpha+3}.
\]
This would leave the case 
\[
\max\left(1,\frac{4\alpha+4}{2\alpha+3+2\theta}\right)\le p\le\frac{4\alpha+4}{2\alpha+3}
\]
unexplored. For this reason, here we only study the case $\theta=0$, that is the {\em partial sums}
of the Fourier series with respect to $\{v_n\}$, obtaining
a few results on the pointwise convergence, and leave the case $\theta>0$ for future studies. 
More precisely, we show that when $\theta=0$ there is a.e. convergence in $L^p((0,\pi/2),A(t)dt)$ if and only if $p>(4\alpha+4)/(2\alpha+3)$
(Theorems \ref{Meaney} and \ref{a.e.}) and discuss the nature of the sets of divergence for
continuous and more regular functions (Theorems \ref{KK}, \ref{Hdim} and \ref{yes})


\section{An abstract equiconvergence theorem}\label{abstract}
Let $\{u_n\}_{n\geq 0}$ be an orthonormal basis of of $L^2([0,\pi/2])$.
We further assume that $\{u_n\}_{n\geq 0}$ satisfies the following properties, where we define $\Delta
u_{n}\left(  x\right)  =u_{n}\left(  x\right)  -u_{n+1}\left(  x\right)  $.

There exists a constant $C>0$ and a positive integer $n_0$ such that

\begin{description}
\item[(P1)] $\sup_{0<x<\pi/2}\left\vert u_{n}\left(
x\right)  \right\vert \leq C,$ $n=0,1,\ldots$

\item[(P2)] There
exists a function $Y_{0}\left(  x\right)  \in L^{\infty}\left(  0,\pi
/4\right]  $ and constants $\nu,$ $\lambda$ such that
\begin{align*}
u_{n}\left(  x\right)  =&\frac{2}{\sqrt{\pi}}\cos\left(  \left(  2n+\nu\right)
x-\lambda\right)  +Y_{0}\left(  x\right)  \frac{2}{nx}\sin\left(  \left(
2n+\nu\right)  x-\lambda\right)\\
&+O\left(  x^{-2}n^{-2}\right)
\end{align*}
uniformly in $\left[  1/n,\pi/4 \right]$,  $n\geq n_{0}$.

\item[(P3)] There exist functions $Z_{1},\ldots,Z_{4}\in L^{\infty}\left(
0,\pi/4\right]  $ such that
\begin{align*}
\Delta u_{n}\left(  x\right)  &=x\left(  Z_{1}\left(  x\right)  \cos\left(
2nx\right)  +Z_{2}\left(  x\right)  \sin\left(  2nx\right)  \right)  \\
&+\frac
{1}{n}\left(  Z_{3}\left(  x\right)  \cos\left(  2nx\right)  +Z_{4}\left(
x\right)  \sin\left(  2nx\right)  \right)  +O\left(  x^{-1}n^{-2}\right)
\end{align*}
uniformly in $\left[1/n,\pi/4\right]  $, $n\geq n_{0}.$


\item[(P4)] There exists $\tau>0$ such that%
\[
\left\vert \Delta u_{n}(x)\right\vert \leq C\left(  \left(  nx\right)
^{\tau}n^{-1}+n^{-2}\right)  .
\]
uniformly in $\left(  0,1/n\right]  ,$ $n\geq n_{0}$.

\item[(P5)] There is a sequence $\left\{  U_{n}\left(  x\right)  \right\}
_{n=0}^{+\infty}$ satisfying the above properties (P1), \ldots, (P4) such that
\[
u_{n}\left(  \pi/2-x\right)  =\left(  -1\right)  ^{n}U_{n}\left(  x\right)
+O\left(  n^{-2}\right)
\] 
uniformly in $x\in\left(  0,\pi/4
\right]$.  The constants $\nu,\,\lambda,\,\tau$ and the functions $Y_0$, $\,Z_1,\ldots,Z_4$
corresponding to the functions $U_n$ may be different from those corresponding to the functions $u_n$.
\end{description}

For any integer $N$, let $\left\{  r_{n,N}\right\}  _{n=0}^{\infty}$ be a
sequence such that, if we define $\Delta  r_{n,N}  =r_{n,N}%
-r_{n+1,N}$, there exists a constant $B$ such that

\begin{description}
\item[(S1)] $\left\vert r_{n,N}\right\vert \leq B$ for all $n,N\geq0$

\item[(S2)] $\sum_{n=0}^{+\infty}\left\vert \Delta  r_{n,N}
\right\vert \leq B$ for all $N\geq0.$
\end{description}

Define the associated multiplier operator and the corresponding kernel by%
\[
T_{N}f\left(  x\right)  =\sum_{n=0}^{+\infty}r_{n,N}\widehat{f}\left(  n\right)
u_{n}\left(  x\right)  ,\quad T_{N}\left(  x,y\right)  =\sum_{n=0}^{+\infty
}r_{n,N}u_{n}\left(  x\right)  u_{n}\left(  y\right)
\]
where%
\[
\widehat{f}\left(  n\right)  =\int_{0}^{\pi/2}f\left(  x\right)  u_{n}\left(
x\right)  dx.
\]
Notice that in particular when $r_{n,N}=1, n\leq N$ and $0$ otherwise, $T_N$ reduces to the partial sum operator.
We consider the cosine basis of $L^{2}\left(  \left[  0,\pi/2\right]  ,dt\right)  $,
\[
\left\{  \sqrt{\frac{4-2\delta_{0n}}{\pi}}\cos\left(  2nx\right)  \right\}
_{n=0}^{\infty},%
\]
where $\delta_{0n}=1$ when $n=0$ and $0$ otherwise.
We define the multiplier operator associated to the sequence $\{r_{n,N}\}_{n\geq 0,N\geq 0}$ with respect to the cosine basis as %
$$
D_{N}f\left(  x\right)    =\frac{2}{\pi}r_{0,N}\int_{0}^{\pi/2}f\left(
y\right)  dy+\frac{4}{\pi}\sum_{n=1}^{+\infty}r_{n,N}  \int_{0}^{\pi
/2}f\left(  y\right)  \cos\left(  2ny\right)  dy  \cos\left(
2nx\right).$$
The corresponding kernel is given by
$$D_{N}\left(  x,y\right) =\frac{2}{\pi}r_{0,N}+\frac{4}{\pi}\sum
_{n=1}^{+\infty}r_{n,N}\cos\left(  2nx\right)  \cos\left(  2ny\right)  .
$$
We prove that for a fixed $f\in L^1(\left[  0,\pi/2\right])$ and $x\in (0,\pi/2)$,  $T_Nf(x)$ converges to $f(x)$ if and only if $D_Nf(x)$ converges to $f(x)$. More precisely we prove the following theorem:
\begin{theorem}
\label{equiconv_L1} Assume that $\{u_n\}_{n=0}^{+\infty}$ satisfies (P1)-(P5), and $\{r_{n,N}\}$ satisfies (S1) and (S2). 
Assume also that for some dense subspace $\Omega\subseteq
L^{1}\left(  \left[  0,\pi/2\right]\right)  $, for
all $x\in\left(  0,\pi/2\right)  $ and for all $g\in\Omega$
\[
\lim_{N\rightarrow+\infty}T_{N}g\left(  x\right)  -D_{N}g\left(  x\right)
=0,
\]
then for all $x\in\left(  0,\pi/2\right)  $ and for all $f\in L^{1}\left(
\left[  0,\pi/2\right]  \right)  $
\[
\lim_{N\rightarrow+\infty}T_{N}f\left(  x\right)  -D_{N}f\left(  x\right)
=0.
\]
Furthermore, if the convergence of $\lim_{N\rightarrow+\infty}T_{N}g\left(
x\right)  -D_{N}g\left(  x\right)  $ is uniform on some set $\Gamma
\subset\left(  0,\pi/2\right)  $ with positive distance from $0$ and from
$\pi/2$ for each $g\in\Omega$, then for all $f\in L^{1}\left(  \left[
0,\pi/2\right]\right)  $
\[
\lim_{N\rightarrow+\infty}T_{N}f\left(  x\right)  -D_{N}f\left(  x\right)  =0
\]
uniformly on $\Gamma.$
\end{theorem}
\subsection{Preliminary results}
In this subsection we will prove some preliminary results needed to prove Theorem \ref{equiconv_L1}.
We start with a technical lemma. It will be convenient to define some
functions which would be needed in our analysis. For any given integers $M\leq N\ $we set%
\begin{align*}
C_{M,N}^{0}\left(  t\right)   &  =\sum_{n=M}^{N}\cos\left(  2nt\right)
,\qquad S_{M,N}^{0}\left(  t\right)  =\sum_{n=M}^{N}\sin\left(  2nt\right) \\
C_{M,N}^{1}\left(  t\right)   &  =\sum_{n=M}^{N}\frac{1}{n}\cos\left(
2nt\right)  ,\qquad S_{M,N}^{1}\left(  t\right)  =\sum_{n=M}^{N}\frac{1}%
{n}\sin\left(  2nt\right)  .
\end{align*}

\begin{lemma}\label{sineexpansion}
For all integers $0\leq M\leq N$ and for all $t\in\mathbb{R}$%
\begin{align*}
C_{M,N}^{0}\left(  t\right)   &  =\frac{\cos\left(  \left(  N+M\right)
t\right)  \sin\left(  \left(  N-M+1\right)  t\right)  }{\sin t},\\
S_{M,N}^{0}\left(  t\right)   &  =\frac{\sin\left(  \left(  N+M\right)
t\right)  \sin\left(  \left(  N-M+1\right)  t\right)  }{\sin t}.
\end{align*}
In particular, all the above functions are bounded by $\min\left(  \left\vert
\sin t\right\vert ^{-1},N-M+1\right)$.
\end{lemma}
We skip the proof of the above lemma as it simply follows by using the well known formulae for sine and cosine expansions.
\begin{lemma}\label{cosinesuminequalities}
There is a positive constant $C'$ such that for all integers $1\leq M\leq N$
and for all $t\in\mathbb{R}$%
\begin{align*}
\left\vert S_{M,N}^{1}\left(  t\right)  \right\vert  &  \leq C',\\
\left\vert C_{M,N}^{1}\left(  t\right)  \right\vert  &  \leq3+\log\left(
1+\frac{\min\left(  \left\vert \sin t\right\vert ^{-1},N-M\right)  }%
{M}\right)  .
\end{align*}
In particular, $C^1_{M,N}(t)$ is uniformly bounded for $|\sin t|\ge cM^{-1}$.
\end{lemma}

\begin{proof}
Let us begin with the inequality $\left\vert S_{1,N}^{1}\left(  t\right)
\right\vert \leq C$. The inequality for $S_{M,N}^{1}$ with $M>1$ follows immediately by writing $S_{M,N}^{1}= S_{1,N}^{1}- S_{1,M-1}^{1}$.
Notice that $S^1_{1,N}(t)$ is the $N$th partial sum of the Fourier series of the sawtooth function $\pi/2-t$, $t\in[0,\pi]$.  Assuming without loss of generality $0<t<\pi/2$,%
\begin{align*}
\sum_{n=1}^{N}\frac{1}{n}\sin\left(  2nt\right)   &  =\int_{0}^{t}2\sum
_{n=1}^{N}\cos\left(  2nu\right)  du\\
&  =\int_{0}^{t}\left(  \sum_{n=-N}^{N}\cos\left(  2nu\right)  -1\right)  du\\
&  =\int_{0}^{t}\frac{\sin\left(  \left(  2N+1\right)  u\right)  }{\sin
u}du-\int_{0}^{t}du.
\end{align*}
If $t\leq1/\left(  2N+1\right)  $, the inequality%
\[
\left\vert \frac{\sin\left(  \left(  2N+1\right)  u\right)  }{\sin
u}\right\vert \leq2N+1
\]
uniformly in $\left(  0,t\right)  $ gives the result. If $1/\left(
2N+1\right)  <t<\pi/2$, then we split the above integral as
\[
\int_{1/\left(  2N+1\right)  }^{t}\frac{\sin\left(  \left(  2N+1\right)
u\right)  }{\sin u}du+\int_{0}^{1/\left(  2N+1\right)  }\frac{\sin\left(
\left(  2N+1\right)  u\right)  }{\sin u}ds-t.
\]
The last two terms are bounded. As for the first one, setting $\psi\left(
u\right)  =1/\sin\left(  u\right)  $, a decreasing function in the interval
$\left[  1/\left(  2N+1\right)  ,t\right]  $, integration by parts gives%
\begin{align*}
&  \left\vert \int_{1/\left(  2N+1\right)  }^{t}\frac{\sin\left(  \left(
2N+1\right)  u\right)  }{\sin u}du\right\vert \\
&  \leq\frac{1}{2N+1}\left(  \left\vert \left(  \cos\left(  \left(
2N+1\right)  t\right)  \right)  \psi\left(  t\right)  \right\vert +\left\vert
\cos\left(  1\right)  \psi\left(  1/\left(  2N+1\right)  \right)  \right\vert\right)\\
&-\frac 1{2N+1}\int_{1/\left(  2N+1\right)  }^{t}\psi^{\prime}\left(  s\right)  ds \\
&  \leq\frac{1}{2N+1}\left(  \frac{3}{\sin\left(  1/\left(  2N+1\right)
\right)  }\right)  \leq C.
\end{align*}
Let us prove the inequality for $C^1_{M,N}$ now. It clearly suffices to assume $M<N$. It follows immediately after summation by parts,%
\begin{align*}
\left\vert C_{M,N}^{1}\left(  t\right)  \right\vert  &  =\left\vert \frac
{1}{N}\sum_{n=M}^{N}\cos\left(  2nt\right)  +\sum_{j=M}^{N-1}\left(  \frac
{1}{j}-\frac{1}{j+1}\right)  \sum_{n=M}^{j}\cos\left(  2nt\right)  \right\vert.
\\
\end{align*}
The sum in $j$ involves $C^0_{M,j}(t)$. After using the estimates for $C^0_{M,j}(t),\,M\leq j\leq N-1$ from Lemma \ref{sineexpansion} we get 
\[
  \left\vert C_{M,N}^{1}\left(  t\right)  \right\vert  
  \leq\frac{N-M+1}{N}  +\sum_{j=M}^{N-1}\left(  \frac{1}{j}-\frac{1}{j+1}\right)
\min\left(  \frac{1}{\left\vert \sin t\right\vert },j-M+1\right). \\
\]
Clearly $(N-M+1)/N$ is bounded by $1$ for all $1\le M< N$. Calling $[\cdot]$ the floor function, and $\chi_D$ the indicator function of the set $D$, the inequality above further simplifies to 
\begin{align*}
\left\vert C_{M,N}^{1}\left(  t\right)  \right\vert &  \leq1+\sum_{j=M}^{\min\left(  \left[  \left\vert \sin t\right\vert
^{-1}+M-1\right]  ,N-1\right)  }\left(  \frac{1}{j}-\frac{1}{j+1}\right)
\left(  j-M+1\right) \\
&  +\chi_{\left[  0,N-1\right]  }\left(  \left[  \left\vert \sin t\right\vert
^{-1}+M\right]  \right)  \sum_{\left[  \left\vert \sin t\right\vert
^{-1}+M\right]  }^{N-1}\left(  \frac{1}{j}-\frac{1}{j+1}\right)  \frac
{1}{\left\vert \sin t\right\vert }\\
&  \leq1+\left(\sum_{j=M}^{\min\left(  \left[  \left\vert \sin t\right\vert
^{-1}+M-1\right]  ,N-1\right)  }\frac{1}{j+1}\right)+\frac{1}{\left\vert \sin
t\right\vert }\frac{1}{ \left[  \left\vert \sin t\right\vert
^{-1}+M\right]   }\\
&  \leq3+\log\left(  1+\frac{\min\left(  \left\vert \sin t\right\vert
^{-1},N-M\right)  }{M}\right)
\end{align*}
\end{proof}

We will now need a few more inequalities involving sine and cosine functions. These inequalities are standard. We are mentioning them just for the sake of completeness.
\begin{definition}
For any given integers $M\leq N\ $we set%
\begin{align*}
C_{M,N}^{0,-}\left(  t\right)   &  =\sum_{n=M}^{N}\left(  -1\right)  ^{n}%
\cos\left(  2nt\right)  ,\qquad S_{M,N}^{0,-}\left(  t\right)  =\sum_{n=M}%
^{N}\left(  -1\right)  ^{n}\sin\left(  2nt\right) \\
C_{M,N}^{1,-}\left(  t\right)   &  =\sum_{n=M}^{N}\frac{\left(  -1\right)
^{n}}{n}\cos\left(  2nt\right)  ,\qquad S_{M,N}^{1,-}\left(  t\right)
=\sum_{n=M}^{N}\frac{\left(  -1\right)  ^{n}}{n}\sin\left(  2nt\right)  .
\end{align*}

\end{definition}

\begin{lemma}\label{zero_minus}
There is a constant $C$ such that for all integers $0\leq M\leq N$ and for all
$t\in\mathbb{R}$,%
\[
\left\vert C_{M,N}^{0,-}\left(  t\right)  \right\vert \leq\frac{1}{\left\vert
\cos t\right\vert },\quad\left\vert S_{M,N}^{0,-}\left(  t\right)  \right\vert
\leq\frac{1}{\left\vert \cos t\right\vert }.
\]

\end{lemma}

\begin{proof}
Observe that%
\begin{align*}
C_{M,N}^{0,-}\left(  t\right)   &  =\frac{1}{2}\sum_{n=M}^{N}\cos\left(
2n\left(  t+\frac{\pi}{2}\right)  \right)  +\frac{1}{2}\sum_{n=M}^{N}%
\cos\left(  2n\left(  \frac{\pi}{2}-t\right)  \right) \\
&  =\frac{1}{2}C_{M,N}^{0}\left(  t+\frac{\pi}{2}\right)  +\frac{1}{2}%
C_{M,N}^{0}\left(  \frac{\pi}{2}-t\right)  ,
\end{align*}
and therefore%
\[
\left\vert C_{M,N}^{0,-}\left(  t\right)  \right\vert \leq\frac{1}{2\left\vert
\sin\left(  t+\frac{\pi}{2}\right)  \right\vert }+\frac{1}{2\left\vert
\sin\left(  \frac{\pi}{2}-t\right)  \right\vert }\leq\frac{1}{\left\vert
\cos\left(  t\right)  \right\vert }.
\]
The estimate for $S_{M,N}^{0,-}\left(  t\right)  $ is similar.
\end{proof}

\begin{lemma}\label{one_minus}
There is a positive constant $C$ such that for all integers $1\leq M\leq N$
and for all $t\in\mathbb{R}$%
\begin{align*}
\left\vert S_{M,N}^{1,-}\left(  t\right)  \right\vert  &  \leq C,\\
\left\vert C_{M,N}^{1,-}\left(  t\right)  \right\vert  &  \leq3+\log\left(
1+\frac{\min\left(  \left\vert \cos t\right\vert ^{-1},N-M\right)  }%
{M}\right)  .
\end{align*}
\end{lemma}

\begin{proof}
These inequalities follow easily from the identities
\begin{align*}
\sum_{n=M}^{N}\frac{1}{n}\left(  -1\right)  ^{n}\sin\left(  2nt\right)   &
=\frac{1}{2}S_{M,N}^{1}\left(  t+\frac{\pi}{2}\right)  +\frac{1}{2}S_{M,N}%
^{1}\left(  t-\frac{\pi}{2}\right) \\
\sum_{n=M}^{N}\frac{1}{n}\left(  -1\right)  ^{n}\cos\left(  2nt\right)   &
=\frac{1}{2}C_{M,N}^{1}\left(  t+\frac{\pi}{2}\right)  +\frac{1}{2}C_{M,N}%
^{1}\left(  t-\frac{\pi}{2}\right)  .
\end{align*}
\end{proof}

\subsection{Pointwise estimates on the kernels}

Recall that $T_N(x,y)$ and $D_N(x,y)$ are the kernels corresponding to the operators $T_N$ and $D_N$ defined in the beginning of this section.  
In order to prove Theorem~\ref{equiconv_L1}, we need some uniform estimates for $x$ in compact subsets of $(0,\pi/2)$ and $y\in(0,\pi/2)$ on the difference of the kernels $T_N(x,y)$ and $D_N(x,y)$. We have the following theorem:
\begin{theorem}
\label{kernel}For any $0<x<\pi/2$ there is a positive constant $K\left(
x\right)  $ such that for any integer $N$ and for all $y\in(0,\pi/2)$
\[
\left\vert T_{N}\left(  x,y\right)  -D_{N}\left(  x,y\right)  \right\vert \leq
K\left(  x\right).
\]
Furthermore, $K\left(  x\right)  \leq C\left(  x^{-1}+\left(  \pi/2-x\right)
^{-1}\right)$.
\end{theorem}

\begin{proof}
We know that $$T_{N}\left(  x,y\right)  =\sum_{n=0}^{+\infty
}r_{n,N}u_{n}\left(  x\right)  u_{n}\left(  y\right)$$ and 
$$ D_{N}(x,y)= \frac{2}{\pi} r_{0,N} +\frac{4}{\pi}\sum_{n=0}^{+\infty} r_{n,N}\cos(2nx)\cos(2ny),$$
where $\{r_{n,N}\}$ are sequences satisfying properties S1 and S2.

Summation by parts on the difference $T_{N}\left(  x,y\right)  -D_{N}\left(  x,y\right)$ gives%
\begin{align*}
&T_{N}\left(  x,y\right)  -D_{N}\left(  x,y\right) \\
  &  =\lim_{M\rightarrow
+\infty}r_{M,N}\sum_{n=0}^{M}\left(  u_{n}\left(  x\right)  u_{n}\left(
y\right)  -\frac{4-2\delta_{0n}}{\pi}\cos\left(  2nx\right)  \cos\left(
2ny\right)  \right) \\
&  +\sum_{j=0}^{M-1}\Delta r_{j,N}\sum_{n=0}^{j}\left(  u_{n}\left(  x\right)
u_{n}\left(  y\right)  -\frac{4-2\delta_{0n}}{\pi}\cos\left(  2nx\right)
\cos\left(  2ny\right)  \right)  .
\end{align*}

By the assumptions on $\left\{  r_{n,N}\right\}  $, it is therefore enough to
show that for any $0<x<\pi/2$ there is a $K\left(  x\right)  $ such that for
any $M$
\[
\left\vert \sum_{n=0}^{M}\left(  u_{n}\left(  x\right)  u_{n}\left(  y\right)
-\frac{4-2\delta_{0n}}{\pi}\cos\left(  2nx\right)  \cos\left(  2ny\right)
\right)  \right\vert \leq K\left(  x\right)  .
\]

So, without loss of generality, we will assume $M=N$ and $r_{n,N}=1$ for
$n\leq N$ and $0$ elsewhere. Thus, we can rewrite%
\[
D_{N}\left(  x,y\right)  =\frac{2}{\pi}C_{0,N}^{0}\left(  x-y\right)  +\frac{2}{\pi}C_{1,N}%
^{0}\left(  x+y\right)  .
\]
Notice that for all $0<y<\pi/2$, by Lemma \ref{sineexpansion},
\[
\left\vert C_{1,N}^{0}\left(  x+y\right)  \right\vert \leq\frac{C}{\left\vert
\sin\left(  x+y\right)  \right\vert }\leq\frac{C}{\left\vert \sin x\right\vert
}+\frac{C}{\left\vert \sin\left(  \pi/2-x\right)  \right\vert }.
\]
So we just need to handle the term
\[
T_N(x,y)-\frac 2\pi C^0_{0,N}(x-y).
\]
For symmetry reasons, we can assume without loss of generality that
$0<x\leq\pi/4$. Assume $n_{1}=n_{1}\left(  x\right)  =\max\left\{
n_{0},\left[  2/x\right]  +1\right\}  $. We can also assume without loss of
generality that $N\geq n_{1}$, and replace $T_{N}\left(  x,y\right)  $ with
the kernel%
\[
T_{n_{1},N}\left(  x,y\right)  =\sum_{n=n_{1}}^{N}u_{n}\left(  x\right)
u_{n}\left(  y\right)  ,
\]
and $D_{N}\left(  x,y\right)  $ with $(2/\pi) C_{n_{1},N}^{0}\left(  x-y\right)
$. Thus, it will be enough to study
\[
T_{n_{1},N}\left(  x,y\right)  -\frac{2}{\pi}C_{n_{1},N}^{0}\left(
x-y\right)  .
\]

\textbf{Case 1} We first consider the case $0<x/2\leq y\leq\pi/4$.
Since we are assuming $n\geq n_{1}=\max\left\{
n_{0},\left[  2/x\right]  +1\right\}$, it follows that $y\geq x/2\ge 1/n$ and both
$u_{n}\left(  x\right)  $ and $u_{n}\left(  y\right)  $ can be expanded
according to property (P2), obtaining, after some trigonometric manipulations%
\begin{align*}
&u_{n}\left(  x\right)  u_{n}\left(  y\right)  \\
&=\left(  \frac{2}{\sqrt{\pi
}}\cos\left(  \left(  2n+\nu\right)  x-\lambda\right)  +Y_{0}\left(  x\right)
\frac{2}{nx}\sin\left(  \left(  2n+\nu\right)  x-\lambda\right)  +O\left(
x^{-2}n^{-2}\right)  \right) \\
&  \times\left(  \frac{2}{\sqrt{\pi}}\cos\left(  \left(  2n+\nu\right)
y-\lambda\right)  +Y_{0}\left(  y\right)  \frac{2}{ny}\sin\left(  \left(
2n+\nu\right)  y-\lambda\right)  +O\left(  y^{-2}n^{-2}\right)  \right) \\
&= I_1 +I_2 +I_3,
\end{align*}
where 
\begin{align*}
I_1& =\frac{2}{\pi}\cos\left(  \nu\left(  x-y\right)  \right)  \cos\left(
2n\left(  x-y\right)  \right)  -\frac{2}{\pi}\sin\left(  \nu\left(  x-y\right)  \right)  \sin\left(
2n\left(  x-y\right)  \right) \\
&  +\frac{2}{\pi}\cos\left(  \nu\left(  x+y\right)  -2\lambda\right)
\cos\left(  2n\left(  x+y\right)  \right) \\
&-\frac{2}{\pi}\sin\left(  \nu\left(  x+y\right)  -2\lambda\right)
\sin\left(  2n\left(  x+y\right)  \right), \\
\\
I_2 &=\left(  \frac{2Y_{0}\left(  x\right)  }{\sqrt{\pi}x}-\frac{2Y_{0}\left(
y\right)  }{\sqrt{\pi}y}\right)  \cos\left(  \nu\left(  x-y\right)  \right)
\frac{1}{n}\sin\left(  2n\left(  x-y\right)  \right)\\
&+\left(  \frac{2Y_{0}\left(  x\right)  }{\sqrt{\pi}x}-\frac{2Y_{0}\left(
y\right)  }{\sqrt{\pi}y}\right)  \sin\left(  \nu\left(  x-y\right)  \right)
\frac{1}{n}\cos\left(  2n\left(  x-y\right)  \right) ,\\
\\
 I_3&= \left(  \frac{2Y_{0}\left(  x\right)  }{\sqrt{\pi}x}+\frac{2Y_{0}\left(
y\right)  }{\sqrt{\pi}y}\right)  \cos\left(  \nu\left(  x+y\right)
-2\lambda\right)  \frac{1}{n}\sin\left(  2n\left(  x+y\right)  \right) \\&+\left(  \frac{2Y_{0}\left(  x\right)  }{\sqrt{\pi}x}+\frac{2Y_{0}\left(
y\right)  }{\sqrt{\pi}y}\right)  \sin\left(  \nu\left(  x+y\right)
-2\lambda\right)  \frac{1}{n}\cos\left(  2n\left(  x+y\right)  \right)
+O\left(  x^{-2}n^{-2}\right)  .
\end{align*}
Summing over $n_1\leq n\leq N$ the above expression, we obtain%
\begin{align*}
&  T_{n_{1},N}\left(  x,y\right)  -\frac{2}{\pi}C_{n_{1},N}^{0}\left(
x-y\right) 
= J_1 + J_2 + J_3 - \frac{2}{\pi}C_{n_{1},N}^{0}\left(  x-y\right),
\end{align*}
where 
\begin{align*}
J_1 &=\frac{2}{\pi}\cos\left(  \nu\left(  x-y\right)  \right)  C_{n_{1},N}%
^{0}\left(  x-y\right)  -\frac{2}{\pi}\sin\left(  \nu\left(  x-y\right)
\right)  S_{n_{1},N}^{0}\left(  x-y\right), \\
\\
  J_2&=\frac{2}{\pi}\cos\left(  \nu\left(  x+y\right)  -2\lambda\right)
C_{n_{1},N}^{0}\left(  x+y\right)  -\frac{2}{\pi}\sin\left(  \nu\left(
x+y\right)  -2\lambda\right)  S_{n_{1},N}^{0}\left(  x+y\right) \\
 &+\left(  \frac{2Y_{0}\left(  x\right)  }{\sqrt{\pi}x}-\frac{2Y_{0}\left(
y\right)  }{\sqrt{\pi}y}\right)  \cos\left(  \nu\left(  x-y\right)  \right)
S_{n_{1},N}^{1}\left(  x-y\right) \\
\\
 J_3&=\left(  \frac{2Y_{0}\left(  x\right)  }{\sqrt{\pi}x}-\frac{2Y_{0}\left(
y\right)  }{\sqrt{\pi}y}\right)  \sin\left(  \nu\left(  x-y\right)  \right)
C_{n_{1},N}^{1}\left(  x-y\right) \\
  &+\left(  \frac{2Y_{0}\left(  x\right)  }{\sqrt{\pi}x}+\frac{2Y_{0}\left(
y\right)  }{\sqrt{\pi}y}\right)  \cos\left(  \nu\left(  x+y\right)
-2\lambda\right)  S_{n_{1},N}^{1}\left(  x+y\right) \\
  &+\left(  \frac{2Y_{0}\left(  x\right)  }{\sqrt{\pi}x}+\frac{2Y_{0}\left(
y\right)  }{\sqrt{\pi}y}\right)  \sin\left(  \nu\left(  x+y\right)
-2\lambda\right)  C_{n_{1},N}^{1}\left(  x+y\right)  +O\left(  x^{-1}\right).
\end{align*}

By Lemma \ref{sineexpansion},  $\left(  \cos\left(  \nu\left(
x-y\right)  \right)  -1\right)  C_{n_{1},N}^{0}\left(  x-y\right)  $ and 
$\sin(\nu(x-y))S_{n_1,N}^0(x-y)$ are uniformly bounded in $N\geq n_{1} $ and in $y\in[x/2,\pi/4]$. 
It also shows that all the other terms in the above sum with 
$S_{n_1,N}^0$ or with $C_{n_1,N}^0$ are uniformly bounded by $C|x|^{-1}$.
The observation that for $x/2\leq y\leq\pi/4$,%
\[
\left\vert \frac{2Y_{0}\left(  x\right)  }{\sqrt{\pi}x}\pm\frac
{2Y_{0}\left(  y\right)  }{\sqrt{\pi}y}
  \right\vert
\leq\frac{6}{\sqrt{\pi}x}\left\Vert Y_{0}\right\Vert _{\infty},
\]
and the estimates of Lemma \ref{cosinesuminequalities}, show that 
the remaining terms in the above sum are bounded by $C|x|^{-1}$ uniformly in $N\geq n_{1}$ and in $y\in[x/2,\pi/4]$
(notice that $\pi/2 \ge x+y\ge x\ge 2n_1^{-1}$, so that Lemma \ref{cosinesuminequalities}
implies that $C^1_{n_1,N}(x+y)$ is bounded).

\textbf{Case 2} We need to check what happens in the case $0<y<x/2.$ Of
course, since%
\[
\left\vert \frac{2}{\pi}C_{n_{1},N}^{0}\left(  x-y\right)  \right\vert
\leq\frac{C}{\left\vert \sin\left(  x-y\right)  \right\vert }\leq\frac{C}%
{\sin\left\vert x/2\right\vert },
\]
we only have to study $\left\vert T_{n_{1},N}\left(  x,y\right)  \right\vert
.$ Set%
\[
\sum_{n=n_{1}}^{N}=\left\{
\begin{array}
[c]{ll}%
\sum_{n=n_{1}}^{\left[  y^{-1}\right]+1  }+\sum_{n=\left[  y^{-1}\right]
+2}^{N} & \text{if }[y^{-1}]+1<N\\
\sum_{n=n_{1}}^{N} & \text{if }N\leq[y^{-1}]+1.
\end{array}
\right.
\]
Call $M=\min\left(  N,[y^{-1}]+1\right)  $. Then summation by parts gives
\[
T_{n_{1},M}\left(  x,y\right)  =\sum_{n=n_{1}}^{M}u_{n}\left(  y\right)
u_{n}\left(  x\right)  =u_{M}\left(  y\right)  \sum_{n=n_{1}}^{M}u_{n}\left(
x\right)  +\sum_{j=n_{1}}^{M-1}\Delta u_{j}\left(  y\right)  \sum_{n=n_{1}%
}^{j}u_{n}\left(  x\right)
\]
Let's begin with the term
\[
u_{M}\left(  y\right)  \sum_{n=n_{1}}^{M}u_{n}\left(  x\right)  .
\]
We know that $u_{M}\left(  y\right)  $ is uniformly bounded in $y$ and $M$ by
property (P1), while $\left\vert \sum_{n=n_{1}}^{M}u_{n}\left(  x\right)
\right\vert \leq C\left\vert x\right\vert ^{-1}$. Indeed, by (P2),%
\begin{align}
&  \sum_{n=n_{1}}^{M}u_{n}\left(  x\right) \label{sum_x}\\
&  =\sum_{n=n_{1}}^{M}  \frac{2}{\sqrt{\pi}}\cos\left(  \left(
2n+\nu\right)  x-\lambda\right)  +Y_{0}\left(  x\right)  \frac{2}{nx}%
\sin\left(  \left(  2n+\nu\right)  x-\lambda\right)  +O\left(  x^{-2}%
n^{-2}\right)  \nonumber\\
&  =\frac{2}{\sqrt{\pi}}\cos\left(  \nu x-\lambda\right)  C_{n_{1},M}%
^{0}\left(  x\right)  -\frac{2}{\sqrt{\pi}}\sin\left(  \nu x-\lambda\right)
S_{n_{1},M}^{0}\left(  x\right) \nonumber\\
&  +Y_{0}\left(  x\right)  \frac{2}{x}\cos\left(  \nu x-\lambda\right)
S_{n_{1},M}^{1}\left(  x\right)  +Y_0(x)\frac{2}{x}\sin\left(  \nu x-\lambda\right)
C_{n_{1},M}^{1}\left(  x\right)  +O\left(  x^{-1}\right)  ,\nonumber
\end{align}
and the desired result follows by Lemma \ref{sineexpansion} and Lemma \ref{cosinesuminequalities}.  Notice in particular that since $\pi/4\ge x\ge 2/n_1$, it follows that 
$C_{n_{1},M}^{1}\left(  x\right)$ is bounded.

Since $M-1\leq y^{-1}$,
it follows that for all $j$ between $n_{1}$ and $M-1$ we
have $y\leq (M-1)^{-1}\leq j^{-1}$ and the following bound given by
property (P4) holds%
\[
\left\vert \Delta u_{j}\left(  y\right)  \right\vert \leq C\left(  \left(
yj\right)  ^{\tau}j^{-1}+j^{-2}\right).
\]
Thus,
\begin{align*}
&  \left\vert \sum_{j=n_{1}}^{M-1}\Delta u_{j}\left(  y\right)  \sum_{n=n_{1}%
}^{j}u_{n}\left(  x\right)  \right\vert  \leq C\sum_{j=n_{1}}^{M-1}\left(  \left(  yj\right)  ^{\tau}%
j^{-1}+j^{-2}\right) 
  \left\vert \sum_{n=n_{1}}^{j}
u_n(x)
\right\vert.
\end{align*}
By \eqref{sum_x},
$
\left\vert \sum_{n=n_{1}}^{j}
u_n(x)
 \right\vert
 =O\left(  x^{-1}\right)  .
$
On the other hand,%
\begin{align*}
\sum_{j=n_{1}}^{M-1}\left(  \left(  yj\right)  ^{\tau}j^{-1}%
+j^{-2}\right)&\leq\left(  \sum_{j=n_{1}}^{M-1}\left(  yj\right)
^{\tau}j^{-1}\right)  +C
 \leq C(yM)^{\tau}+C\leq
C.
\end{align*}
Hence we get that $$|T_{n_{1},M}\left(  x,y\right)|\leq C x^{-1}.$$

It remains to deal with the case $N>[y^{-1}]+1=M$ and the sum
\[
\sum_{n=M+1}^{N}u_{n}\left(  y\right)  u_{n}\left(
x\right).
\]
Replacing $u_{n}\left(  x\right)  $ with the expression given by property (P2)
we have
\begin{align}\label{tough}
&  \sum_{n=M+1}^{N}u_{n}\left(  y\right)  u_{n}\left(
x\right)
=   \sum_{n=M+1}^{N}u_{n}\left(  y\right)  \frac
{2}{\sqrt{\pi}}\cos\left(  \left(  2n+\nu\right)  x-\lambda\right)\\
&+\sum_{n=M+1}^{N}u_{n}\left(  y\right)  Y_{0}\left(
x\right)  \frac{2}{nx}\sin\left(  \left(  2n+\nu\right)  x-\lambda\right)
  +\sum_{n=M+1}^{N}u_{n}\left(  y\right)  O\left(
x^{-2}n^{-2}\right).\nonumber
\end{align}
Clearly, by (P1)
\[
\sum_{n=M+1}^{N}u_{n}\left(  y\right)  O\left(
x^{-2}n^{-2}\right)  =O\left(  x^{-2}y\right)  =O\left(  x^{-1}\right)  .
\]
For the remaining terms, apply summation by parts. For example,%
\begin{align*}
&\sum_{n=M +1}^{N}u_{n}\left(  y\right)  \frac{2}%
{\sqrt{\pi}}\cos\left(  \left(  2n+\nu\right)  x-\lambda\right)  
=u_{N}\left(  y\right)  \sum_{n=M  +1}^{N}\frac{2}%
{\sqrt{\pi}}\cos\left(  \left(  2n+\nu\right)  x-\lambda\right) \\
 &+\sum_{j=M +1}^{N-1}\Delta u_{j}\left(  y\right)
\sum_{n=M +1}^{j}\frac{2}{\sqrt{\pi}}\cos\left(  \left(
2n+\nu\right)  x-\lambda\right)  .
\end{align*}
By (P1), $\left\vert u_{N}\left(  y\right)  \right\vert $ is uniformly bounded
in $y$ and in $N$, and as usual
\[
\left\vert \sum_{n=M
+1}^{N}\frac{2}{\sqrt{\pi}}\cos\left(  \left(  2n+\nu\right)  x-\lambda
\right)  \right\vert \leq C\left\vert x\right\vert ^{-1}.
\] 
Thus we need to
study
\[
\sum_{j=M  +1}^{N-1}\Delta u_{j}\left(  y\right)
\sum_{n=M +1}^{j}\cos\left(  \left(  2n+\nu\right)
x-\lambda\right)  .
\]
For all indices $j$ between $M  +1$ and $N-1$ we have
$j^{-1}<y<x/2<x$ so that for $\Delta u_j\left(  y\right)  $ we use the
expansions given by (P3),
\begin{align*}
\Delta u_{j}\left(  y\right)  &=y\left(  Z_{1}\left(  y\right)  \cos\left(
2jy\right)  +Z_{2}\left(  y\right)  \sin\left(  2jy\right)  \right)  \\
&+\frac{1}{j}\left(  Z_{3}\left(  y\right)  \cos\left(  2jy\right)  +Z_{4}\left(
y\right)  \sin\left(  2jy\right)  \right)
 +O\left(  j^{-2}y^{-1}\right)  .
\end{align*}
Thus, 
\begin{align*}
&  \sum_{j=M+1}^{N-1}\Delta u_{j}\left(  y\right)  \sum_{n=M+1}^{j}\cos\left(
\left(  2n+\nu\right)  x-\lambda\right) \\
&  =\sum_{j=M+1}^{N-1}\left(  y\left(  Z_{1}\left(  y\right)  \cos\left(
2jy\right)  +Z_{2}\left(  y\right)  \sin\left(  2jy\right)  \right)\phantom{\vrule height 3.6ex depth 0pt width 0pt}  \right.\\
&\qquad\qquad\left. +\frac
{1}{j}\left(  Z_{3}\left(  y\right)  \cos\left(  2jy\right)  +Z_{4}\left(
y\right)  \sin\left(  2jy\right)  \right)+O\left(  j^{-2}y^{-1}\right)
\right) \\
&\qquad\times\left(  \cos\left(  \nu x-\lambda\right)  C_{M+1,j}^{0}\left(  x\right)
-\sin\left(  \nu x-\lambda\right)  S_{M+1,j}^{0}\left(  x\right)  \right) \\
&=\sum_{j=M+1}^{N-1}\left(\phantom{\vrule height 3.6ex depth 0pt width 0pt}     y\left(Z_{1}\left(  y\right)  \cos\left(
2jy\right)  +Z_{2}\left(  y\right)  \sin\left(  2jy\right)  \right)  \right.\\
&\qquad\qquad\left.+\frac
{1}{j}\left(  Z_{3}\left(  y\right)  \cos\left(  2jy\right)  +Z_{4}\left(
y\right)  \sin\left(  2jy\right)  \right) +O\left(  j^{-2}y^{-1}\right)
\right)
 \\
&\qquad\times\left(  \cos\left(  \nu x-\lambda\right)  \frac{\cos\left(  \left(
j+M+1\right)  x\right)  \sin\left(  \left(  j-M\right)  x\right)  }{\sin
x}\right.\\
&\qquad\qquad\left.-\sin\left(  \nu x-\lambda\right)  \frac{\sin\left(  \left(  j+M+1\right)
x\right)  \sin\left(  \left(  j-M\right)  x\right)  }{\sin x}\right)
\end{align*}
Simple trigonometric identities reduce the products
\[
{\cos \choose{\sin}}(2jy) {\cos\choose\sin} ((j+M+1)x)\sin((j-M)x)
\]
to linear combinations of terms of the form 
\[
{\cos\choose\sin}(2jy),\quad{\cos\choose\sin}((2j(x\pm y)),
\]
with coefficients that do not depend on $j$ and that are bounded in $x$ and $y$. Summing over $j$,
we obtain that 
\[
\sum_{j=M+1}^{N-1}\Delta u_{j}\left(  y\right)  \sum_{n=M+1}^{j}\cos\left(
\left(  2n+\nu\right)  x-\lambda\right)
\]
equals $(\sin x)^{-1}$ times the linear combination of expressions of the form 
\[
yC_{M+1,N-1}^{0}\left(  t\right)  ,\quad yS_{M+1,N-1}^{0}\left(  t\right)  ,\quad
C_{M+1,N-1}^{1}\left(  t\right)  ,\quad S_{M+1,N-1}^{1}\left(  t\right)  ,
\]
where $t$ can be $x$ or $y$ or $x\pm y$, and with coefficients that are bounded in $x$ and $y$.

By Lemma \ref{sineexpansion} and Lemma \ref{cosinesuminequalities}, all these are
uniformly bounded. In particular, since $M=[y^{-1}]+1$, we have $\pi/2\ge t\ge y\ge (M+1)^{-1}$
and, by Lemma \ref{cosinesuminequalities}, $C^1_{M+1,N-1}(t)$ is uniformly bounded.

Finally, 
\[
\sum_{j=M+1}^{N-1}O(j^{-2}y^{-1})O(|\sin x|^{-1})=O(M^{-1}y^{-1})O(x^{-1})=O(x^{-1}).
\]

It remains to bound the term
\[
\frac 1 x\sum_{n=M+1}^N\frac{u_n(y)}n\sin((2n+\nu)x-\lambda)
\]
in formula \eqref{tough}.
Summation by parts gives
\begin{align*}
&\frac 1{xN} u_N(y)\sum_{n=M+1}^N\sin((2n+\nu)x-\lambda)\\
&+\frac 1x \sum_{j={M+1}}^{N-1}
\left(\frac{u_j(y)}j-\frac{u_{j+1}(y)}{j+1}\right)\sum_{n=M+1}^j\sin((2n+\nu)x-\lambda)
\end{align*}
Since $x>2y>2/N$, we deduce that $(xN)^{-1}$ is bounded and since, as usual, $\sum_{n=M+1}^N\sin((2n+\nu)x-\lambda)=O(x^{-1})$, the first term in the above sum is settled. By {(P1)},
\[
\frac{u_j(y)}j-\frac{u_{j+1}(y)}{j+1}=\frac{\Delta u_j(y)}{j}+O(j^{-2}),
\]
and we are left with
\[
\frac 1x \sum_{j={M+1}}^{N-1}
\frac{\Delta u_j(y)}{j}\sum_{n=M+1}^j\sin((2n+\nu)x-\lambda)
+\frac 1x\sum_{j=M+1}^{N-1}O(j^{-2})O(x^{-1}).
\]
The remainder term gives $O(x^{-2}M^{-1})=O(x^{-2}y)=O(x^{-1})$. The principal part can be treated as for the previous term of \eqref{tough}, noticing that, by (P3), 
\[
\frac{\Delta u_j(y)}{j}=yZ_1(y)\frac{\cos(2jy)}j+yZ_2(y)\frac{\sin(2jy)}j+O(j^{-2}).
\] 

\textbf{Case 3}. We will deal with the case when $y\in\left[ \pi/4,\pi/2\right)  $.
Observe that by property (P5), setting $z=\pi/2-y$ so that $z\in\left(  0,{\pi}/{4}\right]  $,%
\begin{align*}
&  T_{n_{1},N}\left(  x,y\right)  -\frac{2}{\pi}C_{n_{1},N}^{0}\left(
x-y\right) \\
&  =\sum_{n=n_{1}}^{N}\left(  u_{n}\left(  x\right)  u_{n}\left(  y\right)
-\frac{4-2\delta_{0,n}}{\pi}\cos\left(  2nx\right)  \cos\left(  2ny\right)
\right) \\
&  =\sum_{n=n_{1}}^{N}\left(  \left(  -1\right)  ^{n}u_{n}\left(  x\right)
U_{n}\left( z\right)  -\frac{4-2\delta_{0,n}}{\pi}\left(
-1\right)  ^{n}\cos\left(  2nx\right)  \cos\left(  2n\left( z\right)  \right)  \right)+O\left(  1\right),
\end{align*}
and we can proceed as we did in cases 1 and 2, this time using also Lemma \ref{zero_minus} and Lemma \ref{one_minus}.
We leave the details to the reader. 
\end{proof}

\subsection {Proof of Theorem \ref{equiconv_L1}}

We are now in a position to prove Theorem \ref{equiconv_L1}.

\begin{proof} 
Fix a positive $\varepsilon.$ Let $g\in\Omega$ be such that
\[
\left\Vert f-g\right\Vert _{L^{1}\left(  \left[  0,\pi/2\right]\right)  }\leq\frac{\varepsilon}{2K\left(  x\right)  }.
\]
Then%
\begin{align*}
&  \left\vert  T _{N}f\left(  x\right)  -D_{N}f\left(  x\right)
\right\vert \\
&  \leq\left\vert T _{N}\left(  f-g\right)  \left(  x\right)
-D_{N}\left(  f-g\right)  \left(  x\right)  \right\vert +\left\vert
T_{N}g\left(  x\right)  -D_{N}g\left(  x\right)  \right\vert \\
&  \leq\int_{0}^{\pi/2}K\left(  x\right)  \left\vert f\left(  y\right)
-g\left(  y\right)  \right\vert dy+\left\vert T_{N}g\left(  x\right)
-D_{N}g\left(  x\right)  \right\vert \\
&  \leq\frac{\varepsilon}{2}+\left\vert T_{N}g\left(  x\right)  -D_{N}g\left(
x\right)  \right\vert <\varepsilon
\end{align*}
for $N$ sufficiently big so that $\left\vert T_{N}g\left(  x\right)
-D_{N}g\left(  x\right)  \right\vert <\varepsilon/2$. The second part of the
theorem follows from the estimates on $K\left(  x\right)  $ in Theorem
\ref{kernel}.
\end{proof}

\section{Perturbed Jacobi operator in normal form}\label{L_normal}

In this section we consider the operator given by 

\[
\ell u:=-u^{\prime\prime}+\left(  \left(  \alpha^{2}-\frac{1}{4}\right)
\cot^{2}t+\left(  \beta^{2}-\frac{1}{4}\right)  \tan^{2}t-\chi\left(
t\right)  \right)  u
\]
where
$\chi$ is a twice continuously differentiable function on $\mathbb R$, even with respect to $0$ and
$\pi/2$, and $\alpha\ge\beta>-1/2$. We will identify the proper domain where $\ell$ is self-adjoint,
and show that the corresponding orthonormal basis of eigenfunctions satisfies properties (P1) to (P5).

It is clear that the operator $\ell$ has a singularity at $0$ and $\frac{\pi
}{2}$. We will deal with the singularities separately to study the properties of the solution of 
the eigenvalue problem
\begin{equation}
\ell u=\mu u \label{A}
\end{equation}
We know that $\cot t$ behaves like $t^{-1}$ near $0^+.$ We try to approximate $\ell$ by the Bessel equation of order $\alpha$ near $0.$ 
After adding and subtracting $(\alpha^{2}-1/4)t^{-2}$ we get the following:
\begin{equation}
\ell u=-u^{\prime\prime}+\left(  \left(  \alpha^{2}-\frac{1}{4}\right)
t^{-2}-\eta_{0}\left(  t\right)  \right)  u=\mu u, \label{B}%
\end{equation}
where%
\[
\eta_{0}\left(  t\right)  =-\left(  \alpha^{2}-\frac{1}{4}\right)  \left(
\cot^{2}t-\frac{1}{t^{2}}\right)  -\left(  \beta^{2}-\frac{1}{4}\right)
\tan^{2}t+\chi\left(  t\right)
\]
is even with respect to $0$ and in $\mathcal C^{2}\left(  -\pi/2,\pi/2\right)  .$
Again $\tan t$ behaves like $(\pi/2-t)^{-1}$ near $\pi/2^{-}$.
Similarly, adding and subtracting $\left(  \beta^{2}-1/4\right)  \left(
\pi/2-t\right)  ^{-2}$ we get%
\[
\ell u=-u^{\prime\prime}+\left(  \left(  \beta^{2}-\frac{1}{4}\right)  \left(
\frac{\pi}{2}-t\right)  ^{-2}-\eta_{1}\left(  \frac{\pi}{2}-t\right)  \right)
u=\mu u
\]
where%
\[
\eta_{1}\left(  t\right)  =-\left(  \beta^{2}-\frac{1}{4}\right)  \left(
\cot^{2}t-\frac{1}{t^{2}}\right)  -\left(  \alpha^{2}-\frac{1}{4}\right)
\tan^{2}t+\chi\left(  \frac{\pi}{2}-t\right)
\]
is also even with respect to $0$ and in $C^{2}\left( -\pi/2,\pi/2\right)$.
We need asymptotics of solutions of \eqref{A} near $0^+$ and $\pi/2^-$ for obtaining the properties (P1) to (P5) of the corresponding eigenfunctions.
We state the following asymptotic expansion, as found in \cite{BG}. The result is 
classic, see also \cite{FH, ST}.
\begin{theorem}
\label{BG}If we set
\[
X_{0}\left(  t\right)  =\int_{0}^{t}\eta_{0}\left(  s\right)  ds,
\]
then there exists a unique twice differentiable solution $V_{\mu,\alpha}$ of
\eqref{A} such that%

$$\left\vert V_{\mu,\alpha}\left(  t\right)  -\frac{2^{\alpha}\Gamma\left(
\alpha+1\right)  t^{1/2}}{\left(  \sqrt{\mu}\right)  ^{\alpha}}\left(
J_{\alpha}\left(  \sqrt{\mu}t\right)  -\frac{1}{2}X_{0}\left(  t\right)
\frac{J_{\alpha+1}\left(  \sqrt{\mu}t\right)  }{\sqrt{\mu}}\right)
\right\vert $$ $$\leq C\frac{t^{2}\min\left(  1,\left\vert \sqrt{\mu_{n}%
}t\right\vert \right)  ^{\alpha+5/2}}{\left\vert \sqrt{\mu}\right\vert
^{\alpha+5/2}} \label{V_mu_alpha}$$%
uniformly in $t\in\left(  0,\pi/4+\varepsilon\right)  $ and $\mu>1$. 
When
$\alpha=0$, the right hand side above has to be multiplied by the extra
factor $\log\left(  2/\min(1,\left\vert \sqrt{\mu}\right\vert t)\right)  $.
Similarly, if we set%
\[
X_{1}\left(  t\right)  =\int_{0}^{t}\eta_{1}\left(  s\right)  ds,
\]
then there exists a unique twice differentiable solution $W_{\mu,\beta}$ of
\eqref{A} such that
$$
\left\vert W_{\mu,\beta}\left(  \frac{\pi}{2}-t\right)  -\frac{2^{\beta}%
\Gamma\left(  \beta+1\right)  t^{1/2}}{\left(  \sqrt{\mu}\right)  ^{\beta}%
}\left(  J_{\beta}\left(  \sqrt{\mu}t\right)  -\frac{1}{2}X_{1}\left(
t\right)  \frac{J_{\beta+1}\left(  \sqrt{\mu}t\right)  }{\sqrt{\mu}}\right)
\right\vert$$ $$\leq C\frac{t^{2}\min\left(  1,\left\vert \sqrt{\mu}\right\vert
t\right)  ^{\beta+5/2}}{\left\vert \sqrt{\mu}\right\vert ^{\beta+5/2}}
\label{W_mu_beta}%
$$
uniformly in $t\in\left(  0,\pi/4+\varepsilon\right)  $ and $\mu>1$. Again,
when $\beta=0$, the right hand side above has to be multiplied by the extra
factor $\log\left(  2/\min(1,\left\vert \sqrt{\mu}\right\vert t)\right)  .$
\end{theorem}

\begin{proof}
For the proof, we refer the reader to the above mentioned references. 
Regarding existence, we only want to mention that the proof in 
\cite{BG} works under the hypotheses we have here.

Concerning uniqueness, observe that the difference
between two solutions satisfying the above estimate would also be a solution, too small close to $0$.
Indeed it is not difficult to show that there is a second solution of the
equation that goes to zero as $t^{-\alpha+1/2}$ for $\alpha\neq0$, and as
$t^{1/2}\log t$ for $\alpha=0$. Thus the difference between any two different
solutions cannot be $o\left(  t^{\left\vert \alpha\right\vert +1/2}\right)  $
as $t\rightarrow0+.$ On the other hand, the difference of two solutions
satisfying the bounds of the theorem should be $O\left(  t^{\alpha
+9/2}\right)  $ as $t\rightarrow0+$, and this is absurd.
\end{proof}

Recall that $\ell$ is a second order differential operator on $(0,\pi/2)$ with singularities at $0$ and $\frac{\pi}{2}.$ It is easy to check using integration by parts that $\ell$ is a symmetric operator when applied to smooth functions with compact support in $(0,\pi/2).$ We want to extend $\ell$ as a self-adjoint operator in $L^2((0,\pi/2))$ to obtain an orthonormal basis of $L^2((0,\pi/2))$ such that the basis elements are the eigenfunctions of $\ell.$ Once we get an orthogonal expansion we can talk about the convergence of partial sum operator with respect to that expansion.

Niessen and Zettl \cite{NZ} have proved the existence of self-adjoint extensions of a general class of Sturm-Liouville operators on $(a,b)$ with singularities at $a$ and $b.$ They classify all the possible self-adjoint extensions of non-oscillatory Sturm-Liouville operators. The self-adjoint extensions depend on the boundary conditions at the end points. In order to apply the result of Niessen and Zettl we have to specify the boundary conditions in such a way that when $B\equiv 1$ the eigenfunctions are Jacobi polynomials.  Before going further we will give some new definitions. For further details the reader is referred to \cite{NZ}.

 \begin{definition}A differential equation is oscillatory at an (or
both) endpoint(s) if the zeros of one, and hence every, non-trivial real valued solution
accumulate at the endpoint(s), otherwise it is called non-oscillatory.
\end{definition}
From the asymptotics of the solution of the above differential equation
\eqref{A} in Theorem \ref{BG}, it is clear that $\ell$ is non oscillatory at both endpoints.

We say $u$ is a principal solution of \eqref{A} at $0$ or $\frac{\pi}{2}$ if
for any other real valued solution $y$ of \eqref{A} which is not a multiple of
$u$ we have $u(t)=o(y(t))$ as $t\rightarrow0$ or $\frac{\pi}{2}$, otherwise we
say it is non principal.

By Theorem \ref{BG} it is easy to see that $V_{\mu,\left\vert \alpha\right\vert
}\left(  t\right)  $ as defined above satisfies
\[
\int_{0}^{\varepsilon}\frac{1}{\left\vert V_{\mu,\left\vert \alpha\right\vert
}\left(  t\right)  \right\vert ^{2}}dt\sim\int_{0}^{\varepsilon}\frac
{1}{t^{2\left\vert \alpha\right\vert +1}}dt=+\infty.
\]

So by Theorem 2.2 page 548 in \cite{NZ}, $V_{\mu,\left\vert \alpha\right\vert
}\left(  t\right)  $ is a principal solution at $0.$ Similarly we can show
that $W_{\mu,\left\vert \beta\right\vert }\left(  t\right)  $ is a principal
solution at $\pi/2$.

By inspection, one can easily see that a second solution, linearly independent
of $V_{\mu,\left\vert \alpha\right\vert }\left(  t\right)  $, is given by%
\[
\widetilde{V}_{\mu,\alpha}\left(  t\right)  =V_{\mu,\left\vert \alpha\right\vert
}\left(  t\right)  \int_{t}^{\varepsilon}\frac{1}{V_{\mu,\left\vert
\alpha\right\vert }\left(  u\right)  ^{2}}du.
\]
Notice that, as $t\rightarrow0+,$
\[
\widetilde{V}_{\mu,\alpha}\left(  t\right)  \sim V_{\mu,\left\vert \alpha
\right\vert }\left(  t\right)  \int_{t}^{\varepsilon}\frac{1}{u^{2\left\vert
\alpha\right\vert +1}}du\sim\left\{
\begin{array}
[c]{l}%
t^{-\left\vert \alpha\right\vert +1/2}\text{ if }\alpha\neq0\\
-t^{1/2}\log t\text{ if }\alpha=0.
\end{array}
\right.
\]
(similar estimates hold for $\widetilde{W}_{\mu,\beta}\left(  t\right)  $ near
$\pi/2$).

If $\alpha<1$ then $\ell$ is in the limit circle case at $0$, i.e. every
solution of $\ell u=\mu u$ is in $L^{2}\left( ( 0,\varepsilon)\right)  $, and if
$\beta<1$ then $\ell$ is in the limit circle case at $\pi/2,$ i.e. every solution
of $\ell u=\mu u$ is in $L^{2}\left(  (\pi/2-\varepsilon,\pi/2)\right)  $. Else,
by definition we say that $\ell$ is in the limit point case.
It is known  \cite[\S 19.4 Theorem 4]{Naimark} that this classification
is independent of $\mu$ (this was anyway transparent here, by the above considerations on the linearly independent solutions $V_{\mu,|\alpha|}$ and $\widetilde V_{\mu,\alpha}$). 


The unbounded operator $\ell:\mathcal{D\subset}L^{2}\rightarrow L^{2}$ is
defined in
\[
\mathcal{D=}\left\{  f,\,f^{\prime}\in\mathrm{AC}_{\mathrm{loc}},\,f,\ell\left(
f\right)  \in L^{2}\left(  0,\pi/2\right)  \right\},
\]
where $\mathrm{AC}_{\mathrm{loc}}$ is the class of absolutely continuous
functions on all the compact subintervals of $(0,\pi/2)$.
It is known (see \cite{Naimark} or \cite[page 549]{NZ}) that
$\mathcal D$ is dense in $L^2(0,\pi/2)$. 

\begin{proposition}
\label{boundary conditions}The operator $\tilde{\ell}$ obtained as the
restriction of $\ell$ to the domain%
\[
\mathcal{D}(\tilde{\ell})=\left\{
\begin{array}
[c]{ll}%
\left\{  y\in\mathcal{D}:\left[  y,V_{\mu,\alpha}\right] \left(
0\right)  =\left[  y,W_{\widetilde{\mu},\beta}\right]\left(
\pi/2\right)  =0\right\}  & \text{if }-1/2<\beta\leq\alpha<1\\
\left\{  y\in\mathcal{D}:\left[  y,W_{\mu,\beta}\right] \left(
\pi/2\right)  =0\right\}  & \text{if }-1/2<\beta<1\leq\alpha\\
\mathcal{D} & \text{if }1\leq\beta\leq\alpha
\end{array}
\right.
\]
is self-adjoint and bounded from below. Here%
\begin{align*}
\left[  y,u\right]\left(  0\right)   &  =\lim_{t\rightarrow0+}\left(
y\left(  t\right)  \overline{u^{\prime}\left(  t\right)  }-y^{\prime}\left(
t\right)  \overline{u\left(  t\right)  }\right) \\
\left[  y,u\right] \left(  \pi/2\right)   &  =\lim_{t\rightarrow
\pi/2-}\left(  y\left(  t\right)  \overline{u^{\prime}\left(  t\right)
}-y^{\prime}\left(  t\right)  \overline{u\left(  t\right)  }\right)  ,
\end{align*}
and the domain is independent of the choice of $\mu$, $\widetilde{\mu}$,
and the function $\chi$.
All eigenvalues are simple and can be ordered by 
\[
\mu_{0}<\mu_{1}<\ldots<\mu_{n}<\ldots
\]
with $\mu_{n}\rightarrow+\infty$. More precisely,
\[
\mu_n=4n^2+4(\alpha+\beta+1)n+O(1),\quad\text{ as }n\to+\infty. 
\]
\end{proposition}

\begin{proof}
Assume first $0\le\beta\le\alpha$ so that both $V_{\mu,\alpha}$ and $W_{\widetilde\mu,\beta}$ are principal solutions, and consider the preminimal symmetric operator $\ell_0'$ defined on
\[
\mathcal D_0'=\{y\in\mathcal D:y \text{ has compact support in } (0,\pi/2)\}
\]
by
$\ell_0'(y)=\ell(y)$. $\mathcal D_0'$ is dense in $L^2((0,\pi/2))$ and by Theorem 4.2 in \cite{NZ}, 
$\ell_0'$ is bounded below.  Theorem 4.2 and Corollary 4.1 in \cite{NZ} guarantee that the Friedrichs extension
of $\ell_0'$ coincides with $\tilde\ell$. It is well known (see Definition 3.1 in \cite{NZ} and the subsequent 
comments) that the Friedrichs extension of a densely defined symmetric bounded below operator
is self-adjoint and bounded below with the same bound.

Assume now $-1/2<\beta<0\le\alpha$ so that $V_{\mu,\alpha}$ is a principal solution, but $W_{\widetilde\mu,\beta}$ is not. This also means that $\ell$ is in the limit circle case at $\pi/2$.
Then, by Theorem 4.4 in \cite{NZ}, the operator $\ell_1$ defined in 
\[
\mathcal D_1=\{y\in\mathcal D:[y,W_{\beta,\widetilde\mu}](\pi/2)=0\text{ and } y \text{ is identically } 0 \text{ near } 0\}
\]
by $\ell_1(y)=\ell(y)$, 
is a symmetric operator in $L^2(0,\pi/2)$ which is bounded below, and its Friedrichs extension is 
$\tilde \ell$ as defined in the statement of the proposition.

Finally, when $-1/2<\beta\le\alpha<0$, then both $V_{\mu,\alpha}$ and $W_{\widetilde\mu,\beta}$ are non principal solutions,  and $\ell$ is in the limit circle case at both endpoints. Thus, $\tilde \ell$ is self-adjoint and bounded below by Theorem 5.1 in \cite{NZ}, with the ``separated boundary conditions''
(5.21) and (5.22) with $A_1=B_1=0$ and $A_2=B_2=1$.

It is immediate that the subspace $\mathcal D$ does not depend on $\mu$ or $\widetilde \mu$, and that
if $y\in L^2((0,\pi/2))$, since $\chi$ is bounded, then $\ell y\in L^2((0,\pi/2))$ if and only if
\begin{equation*}
J y:=-y^{\prime\prime}+\left(  \left(  \alpha^{2}-\frac{1}{4}\right)
\cot^{2}t+\left(  \beta^{2}-\frac{1}{4}\right)  \tan^{2}(t)  \right)  y\in L^2((0,\pi/2)).%
\end{equation*}
Thus  $\mathcal D$ does not depend on $\chi$ either.
Let us show that the boundary conditions do not depend on $\mu$, $\widetilde\mu$ nor $\chi$.
Let us focus on the boundary condition at $0$. If $0\le\alpha<1$, then $V_{\mu,\alpha}$ is a principal solution and, by Theorem 4.3 in \cite{NZ} the condition $[y,V_{\mu,\alpha}](0)=0$ is equivalent to 
\[
\lim_{t\to0+}\frac{y(t)}{\widetilde V_{\mu,\alpha}(t)}=0 \quad\iff\quad \left\{
\begin{array}
[c]{l}%
\lim_{t\to0+}\dfrac{y(t)}{t^{- \alpha+1/2}}=0,\text{ if }0<\alpha<1\\
\lim_{t\to0+}\dfrac{y(t)}{-t^{1/2}\log t}=0,\text{ if }\alpha=0,
\end{array}
\right.
\]
and this is independent of $\mu$ and $\chi$.
When $-1/2<\alpha<0$, then we have to proceed differently. The asymptotic expansion 
of $V_{\mu,\alpha}$ in Theorem \ref{BG}, and of its derivative (see the original theorem,
for example in \cite{BG}), guarantee that, calling
\[
V(t)=\frac{V_{\mu_1,\alpha,\chi_1}(t)}{V_{\mu_2,\alpha,\chi_2}(t)},
\]
then $V(t)=1+O(t^2)$ and $V'(t)=O(t)$ as $t\to0+$, where the dependence 
on $\mu_1, \mu_2, \chi_1,\chi_2$ appears only in the remainders.
Assume that $[y,V_{\mu_2,\alpha,\chi_2}](0)=0$. Then
\begin{align*}
&y(t)V_{\mu_1,\alpha,\chi_1}'(t)-y'(t)V_{\mu_1,\alpha,\chi_1}(t)=
-V_{\mu_1,\alpha,\chi_1}^2(t)\left(\frac{y}{V_{\mu_1,\alpha,\chi_1}}\right)'(t)\\
&=-V^2(t)V_{\mu_2,\alpha,\chi_2}^2(t)\left(\frac1{V(t)}\frac{y}{V_{\mu_2,\alpha,\chi_2}}\right)'(t)\\
&=V'(t)V_{\mu_2,\alpha,\chi_2}(t)y(t)-V(t)V_{\mu_2,\alpha,\chi_2}^2(t)\left(\frac{y}{V_{\mu_2,\alpha,\chi_2}}\right)'(t)
\end{align*}
The initial assumption implies that 
\[
\lim_{t\to 0+}V_{\mu_2,\alpha,\chi_2}^2(t)\left(\frac{y}{V_{\mu_2,\alpha,\chi_2}}\right)'(t)=0.
\]
Since $V$ is bounded near $0$ it only remains to show that
$
V'(t)V_{\mu_2,\alpha,\chi_2}(t)y(t)\to 0
$
 as $t\to 0+$. Since 
\[
\left|\left(\frac{y}{V_{\mu_2,\alpha,\chi_2}}\right)'(t)\right|\le Ct^{-2\alpha-1}
\]
with $\alpha<0$, it follows easily that ${y}/{V_{\mu_2,\alpha,\chi_2}}$ is bounded near $0$,
so that 
\[
|V'(t)V_{\mu_2,\alpha,\chi_2}(t)y(t)|\le Ct^{2\alpha+2}\to 0
\]
as $t\to0+$.

The boundary conditions guarantee that all eigenvalues are simple. Indeed, if $0\le \beta\le\alpha$,
then the boundary conditions can be written as (Theorem 4.3 in \cite{NZ}) 
\[
\lim_{t\to0+}\frac{y(t)}{\widetilde V_{\mu,\alpha}(t)}=\lim_{t\to\pi/2-}\frac{y(t)}{\widetilde W_{\widetilde\mu,\beta}(t)}=0.
\]
This shows that if $y$ is an eigenfunction corresponding to $\mu$, then $y$ is a multiple of the principal
solution $V_{\mu,\alpha}$. Thus, the corresponding eigenspace has dimension $1$. The same argument works when $-1/2<\beta<0\le\alpha$ too. If $-1/2<\beta\le\alpha<0$, the the result follows from Theorem 5.3 in \cite{NZ}.

The orthogonality of eigenfunctions corresponding to different eigenvalues, and the separability of 
$L^2((0,\pi/2))$, guarantee that the eigenvalues form a (bounded below) countable subset of $\mathbb R$.

Finally, we know (\cite[page 1544]{DS}) that for $n\geq0$,
\begin{equation}
\mu_{n}=\sup_{H\in S^{\left(  n\right)  }}\inf_{\substack{\left(  u,H\right)
=0\\u\neq0,u\in\mathcal{D}( \tilde\ell) }}\frac{\left(  u,\ell u\right)
}{\left(  u,u\right)  }, \label{eigen}%
\end{equation}
where $S^{(n)}$ denotes the family of all $n$-dimensional subspaces of
$L^{2}\left(  (0,\pi/2)\right)  $. Recall also that $\chi$ is
bounded on $(0,\pi/2)$,
\[
m\leq-\chi\left(  t\right)  \leq M.
\]
The Jacobi operator%
\[
Jv=-u^{\prime\prime}+\left(  \left(  \alpha^{2}-\frac{1}{4}\right)  \cot
^{2}t+\left(  \beta^{2}-\frac{1}{4}\right)  \tan^{2}t\right)  u
\]
on the domain $\mathcal{D}( \tilde \ell)  $ defined as above is self-adjoint, the eigenfunctions are the Jacobi polynomials%
\[
\sin^{\alpha+1/2}t\cos^{\beta+1/2}tP_{n}^{\alpha,\beta}\left(  \cos\left(
2t\right)  \right)  ,\quad n\geq0
\]
with eigenvalues%
\[
\mu_{n}^{J}=\left(  2n+1\right)  ^{2}+2\left(  2n+1\right)  \left(
\alpha+\beta\right)  +2\alpha\beta+\frac{1}{2}%
\]
(see Theorem 4.2.2, page 61, or 4.24.2, page 67, in \cite{Sz}. Also observe that these polynomials satisfy the boundary conditions by Theorem 8.21.12, page 197 in \cite{Sz}),
and therefore comparing
\[
\ell u=Ju-\chi u
\]
with
\[
Ju+mu,\quad Ju+Mu
\]
in \eqref{eigen}, we obtain%
\[
\mu_{n}^{J}+m\leq\mu_{n}\leq\mu_{n}^{J}+M.
\]
\end{proof}

Let us denote the normalized eigenfunctions of $\tilde{\ell}$ by $\{u_{n}%
\}_{n\geq0}$. The self-adjointness of $\tilde{\ell}$ guarantees the orthogonality
of $\{u_{n}\}_{n\geq0}$ and the spectral theorem for a self-adjoint operator
on a Hilbert space gives us that $\{u_{n}\}_{n\geq0}$ is a basis of
$L^{2}((0,\pi/2))$. Concerning the eigenvalues, it may be convenient to
rewrite the expansion in the above proposition in a slightly different form.
When $n$ is big enough, $\mu_{n}>1$ and as $n\rightarrow+\infty$%
\begin{equation}
\sigma_n:=\sqrt{\mu_{n}}=2n+1+\alpha+\beta+O\left(  \frac{1}{n}\right)  .
\label{asymptotic_eigenvalue_one}%
\end{equation}

\subsection{Asymptotics of the eigenfunctions}

For $n\geq0$, let $\{u_{n}\}  $ be the eigenfunctions of $\tilde\ell$ associated
with the eigenvalue $\mu_{n}$, with $L^{2}$ norm equal to $1.$
Once we have got an orthogonal expansion of $\tilde\ell$ we need some estimates for the eigenfunctions to check if they satisfy the properties (P1) to (P5).
By Proposition \ref{boundary conditions}, it follows that there exist
constants $c_{n}$ and $d_{n}$ such that for all $t\in\left(  0,\pi/2\right)
$,
\begin{equation}\label{identity}
u_{n}\left(  t\right)  =c_{n}V_{\mu_{n},\alpha}\left(  t\right)  =d_{n}%
W_{\mu_{n},\beta}\left(  t\right)  .
\end{equation}
This is clearly true for $\alpha<1$ due to the boundary conditions of
Proposition \ref{boundary conditions}, and in the other cases it follows from
the uniqueness of $V_{\mu_{n},\alpha}$ among the solutions of the equation
$\ell u=\mu_{n}u$ which are in $L^{2}\left(  \left(  0,\varepsilon\right)
\right)  $, and of $W_{\mu_{n},\beta}$ in $L^{2}\left(  \pi/2-\varepsilon
,\pi/2\right)  .$

Recall the asymptotic expansions, given  in
Theorem \ref{BG}, for $V_{\mu_{n},\alpha}$ and $W_{\mu_{n},\beta}$ in the intervals $\left(  0,\pi
/4+\varepsilon\right)  $ and $\left(  \pi/4-\varepsilon,\pi/2\right)  )$,
respectively. We will further simplify the expansions in the next lemma. 

\begin{lemma}
\label{V_expansion}We have%
\begin{align*}
V_{\mu_{n},\alpha}\left(  t\right)   &  =\left(  \frac{2}{\pi}\right)
^{1/2}\frac{2^{\alpha}\Gamma\left(  \alpha+1\right)  }{\sigma_{n}^{\alpha
+1/2}}\left(  \cos\left(  \sigma_{n}t-\frac{\alpha\pi}{2}-\frac{\pi}%
{4}\right)  \right. \\
&  \left.  -\left(   \alpha^{2}-1/4  +tX_{0}\left(  t\right)
\right)  \frac{1}{2\sigma_{n}t}\sin\left(  \sigma_{n}t-\frac{\alpha\pi}%
{2}-\frac{\pi}{4}\right)  +O\left( (  \sigma_{n}t)
^{-2}  \right) \right) ,
\end{align*}
uniformly in $\left[  \sigma_{n}^{-1},\pi/4+\varepsilon\right)  $ for $n$
sufficiently big. Similarly,%
\begin{align*}
&W_{\mu_{n},\beta}\left(  t\right)  \\
 &  =\left(  \frac{2}{\pi}\right)
^{1/2}\frac{2^{\beta}\Gamma\left(  \beta+1\right)  }{\sigma_{n}^{\beta+1/2}%
}\left(  \cos\left(  \sigma_{n}\left(  \frac{\pi}{2}-t\right)  -\frac{\beta
\pi}{2}-\frac{\pi}{4}\right)  \right. \\
&  -\left(  \beta^{2}-1/4  +\left(  \frac{\pi}{2}-t\right)
X_{1}\left(  \frac{\pi}{2}-t\right)  \right)  \frac{1}{2\sigma_{n}\left(
\frac{\pi}{2}-t\right)  }\sin\left(  \sigma_{n}\left(  \frac{\pi}{2}-t\right)
-\frac{\beta\pi}{2}-\frac{\pi}{4}\right) \\
&  +\left.  O\left(  \left( \sigma_{n}\left(  \frac{\pi}{2}-t\right)
\right) ^{-2}\right)  \right)  ,
\end{align*}
uniformly in $\left(  \pi/4-\varepsilon,\pi/2-\sigma_{n}^{-1}\right]  $ for
$n$ sufficiently big.
\end{lemma}

\begin{proof}
It follows directly from Theorem \ref{BG} and the asymptotic expansion of
Bessel functions (see \cite{Watson} page 199).
\end{proof}

For proving the properties (P1) to (P5) for the eigenfunctions $\{u_n\}_{n\geq0}$ we need asymptotic estimates of the constants $c_n$ and $d_n$ in \eqref{identity} for $n$ large. The following lemmas give us the desired expansion of $c_n$ and $d_n.$
\begin{lemma}
\label{ratio}The following estimate holds%
\[
\frac{d_{n}}{c_{n}}=\left(  -1\right)  ^{n}\frac{2^{\alpha-\beta}\Gamma\left(
\alpha+1\right)  }{\Gamma\left(  \beta+1\right)  }\sigma_{n}^{\beta-\alpha
}\left(  1+O\left(  n^{-2}\right)  \right)  ,\text{ as }n\rightarrow+\infty.
\]

\end{lemma}

\begin{proof}
Let
\[
t_{n}=\left\{
\begin{array}
[c]{ll}%
\dfrac{\pi}{4}+\dfrac{\left(  \alpha-\beta\right)  \pi}{4\sigma_{n}} &
\text{if }n\text{ is even}\\
\\
\dfrac{\pi}{4}+\dfrac{\left(  \alpha-\beta+2\right)  \pi}{4\sigma_{n}} &
\text{if }n\text{ is odd,}%
\end{array}
\right.
\]
and replace $t$ with $t_{n}$ in the asymptotic expansion of $V_{\mu_{n}%
,\alpha}\left(  t\right)  $ in $W_{\mu_{n},\beta}\left(  t\right)  $ in Lemma
\ref{V_expansion}.  By 
\eqref{asymptotic_eigenvalue_one}, if $n$ is even, then%
\begin{align*}
\cos\left(  t_{n}\sigma_{n}-\frac{\alpha}{2}\pi-\frac{1}{4}\pi\right)
&= \cos\left(  \dfrac{\pi}{2}n+O\left(  n^{-1}\right)  \right)\\
\cos\left(  \left(  \frac{\pi}{2}%
-t_{n}\right)  \sigma_{n}-\frac{\beta}{2}\pi-\frac{1}{4}\pi\right)
& =\cos\left(  \dfrac{\pi}{2}n+O\left(  n^{-1}\right)  \right)  .
\end{align*}
and if $n$ is odd then%
\begin{align*}
\cos\left(  t_{n}\sigma_{n}-\frac{\alpha}{2}\pi-\frac{1}{4}\pi\right)   &
 =\cos\left(  \dfrac{\pi}{2}n+O\left(
n^{-1}\right)  +\frac{1}{2}\pi\right) \\
\cos\left(  \left(  \frac{\pi}{2}-t_{n}\right)  \sigma_{n}-\frac{\beta}{2}%
\pi-\frac{1}{4}\pi\right)   & 
=\cos\left(  \dfrac{\pi}{2}n+O\left(  n^{-1}\right)  -\frac{1}{2}\pi\right).
\end{align*}
In particular, $\sin(t_n\sigma_n-\alpha\pi/2-\pi/4)=O(n^{-1})$
and $\sin((\pi/2-t_n)\sigma_n-\beta\pi/2-\pi/4)=O(n^{-1})$, so that
\begin{align*}
V_{\mu_{n},\alpha}\left(  t_{n}\right)   &  =\frac{2^{\alpha+1/2}\Gamma\left(
\alpha+1\right)  }{\pi^{1/2}\sigma_{n}^{\alpha+1/2}}\left(  \cos\left(
\sigma_{n}t_{n}-\frac{\alpha\pi}{2}-\frac{\pi}{4}\right)  \right. \\
&  \left.  -\left(  \left(  \alpha^{2}-1/4\right)  +t_{n}X_{\alpha}\left(
t_{n}\right)  \right)  \frac{1}{2\sigma_{n}t_{n}}\sin\left(  \sigma_{n}%
t_{n}-\frac{\alpha\pi}{2}-\frac{\pi}{4}\right)  +O\left(
\sigma_{n} ^{-2}\right)  \right) \\
&  =\frac{2^{\alpha+1/2}\Gamma\left(  \alpha+1\right)  }{\pi^{1/2}\sigma
_{n}^{\alpha+1/2}}\left(  \cos\left(  \sigma_{n}t_{n}-\frac{\alpha\pi}%
{2}-\frac{\pi}{4}\right)  +O\left( \sigma_{n}
^{-2}\right)  \right)  ,
\end{align*}
and similarly%
\begin{align*}
W_{\mu_{n},\beta}\left(  t_{n}\right)   
 =\frac{2^{\beta+1/2}\Gamma\left(  \beta+1\right)  }{\pi^{1/2}\sigma
_{n}^{\beta+1/2}}\left(  \cos\left(  \sigma_{n}\left(  \frac{\pi}{2}%
-t_{n}\right)  -\frac{\beta\pi}{2}-\frac{\pi}{4}\right)  +O\left( 
\sigma_{n} ^{-2}\right)  \right).
\end{align*}
Thus, by \eqref{identity} and the above computations,%
\begin{align*}
\frac{d_{n}}{c_{n}}&=\frac{V_{\mu_{n},\alpha}\left(  t_{n}\right)  }{W_{\mu
_{n},\beta}\left(  t_{n}\right)  }\\
&=\frac{2^{\alpha-\beta}\Gamma\left(
\alpha+1\right)  }{\Gamma\left(  \beta+1\right)  }\frac{\cos\left(
t_{n}\sigma_{n}-\frac{\alpha}{2}\pi-\frac{1}{4}\pi\right)  }{\cos\left(
\left(  \frac{\pi}{2}-t_{n}\right)  \sigma_{n}-\frac{\beta}{2}\pi-\frac{1}%
{4}\pi\right)  } \sigma_{n}  ^{\beta-\alpha}+O\left(  
\sigma_{n}  ^{\beta-\alpha-2}\right) \\
&=
\left(  -1\right)  ^{n}\frac{2^{\alpha-\beta}\Gamma\left(
\alpha+1\right)  }{\Gamma\left(  \beta+1\right)  } \sigma_{n}
^{\beta-\alpha}\left(  1+O\left( \sigma_{n} ^{-2}\right)
\right).
\end{align*}
\end{proof}

\begin{lemma}
\label{coefficients}The following estimates hold for $n\rightarrow+\infty$%
\begin{align*}
c_{n}  &  =\frac{\sigma_{n}^{\alpha+1/2}}{2^{\alpha-1/2}\Gamma\left(
\alpha+1\right)  }\left(  1+O\left(  \frac{1}{n^{2}}\right)  \right) \\
d_{n}  &  =\left(  -1\right)  ^{n}\frac{\sigma_{n}^{\beta+1/2}}{2^{\beta
-1/2}\Gamma\left(  \beta+1\right)  }\left(  1+O\left(  \frac{1}{n^{2}}\right)
\right)  .
\end{align*}

\end{lemma}

\begin{proof}
We know that%
\begin{equation}
1=\int_{0}^{\pi/2}\left\vert u_{n}\left(  t\right)  \right\vert ^{2}%
dt=c_{n}^{2}\left(  \int_{0}^{\pi/4}\left\vert V_{\mu_{n},\alpha}\left(
t\right)  \right\vert ^{2}dt+\int_{\pi/4}^{\pi/2}\frac{d_{n}^{2}}{c_{n}^{2}%
}\left\vert W_{\mu_{n},\beta}\left(  t\right)  \right\vert ^{2}dt\right)  .
\label{norm one}%
\end{equation}
Observe that for $n$ big, using Theorem \ref{BG} and the well known boundedness of the function $\sqrt{x}J_\alpha(x)$ for $x>0$ and $\alpha>-1/2$, 
\begin{align}
&  \int_{0}^{\pi/4}\left\vert V_{\mu_{n},\alpha}\left(  t\right)  \right\vert
^{2}dt\label{first half}\\
&=\left(  \frac{2^{\alpha}\Gamma\left(  \alpha+1\right)  }{\sigma
_{n}^{\alpha}}\right)  ^{2}\nonumber\\
&\times  \int_{0}^{\pi/4}\left(  t^{1/2}J_{\alpha}\left(  \sigma_{n}t\right)
-\frac{1}{2}t^{1/2}X_{0}\left(  t\right)  \frac{J_{\alpha+1}\left(  \sigma
_{n}t\right)  }{\sigma_{n}}+O\left(  \frac{t^{2}}{\sigma
_{n} ^{5/2}}\right)  \right) ^{2}dt\nonumber\\
&  =\left(  \frac{2^{\alpha}\Gamma\left(  \alpha+1\right)  }{\sigma
_{n}^{\alpha}}\right)  ^{2}\int_{0}^{\pi/4}tJ_{\alpha}^{2}\left(  \sigma
_{n}t\right)  dt\nonumber\\
&-\left(  \frac{2^{\alpha}\Gamma\left(  \alpha+1\right)
}{ \sigma_{n} ^{\alpha}}\right)  ^{2}\frac{1}{\sigma_{n}%
}\int_{ \sigma_{n}  ^{-1}}^{\pi/4}X_{0}\left(  t\right)
J_{\alpha+1}\left(  \sigma_{n}t\right)  J_{\alpha}\left(  \sigma_{n}t\right)
tdt  +O\left(  \frac{1}{\sigma_{n}^{2\alpha+3}}\right). \nonumber
\end{align}
Notice that%
\[
\int_{ \sigma_{n}  ^{-1}}^{\pi/4}X_{0}\left(  t\right)
J_{\alpha+1}\left(  \sigma_{n}t\right)  J_{\alpha}\left(  \sigma_{n}t\right)
tdt=O\left(  \sigma_{n} ^{-2}\right)  .
\]
Indeed, by the asymptotic expansion of Bessel functions (see \cite[page 199]{Watson})%
\begin{align*}
&  \int_{  \sigma_{n}  ^{-1}}^{\pi/4}X_{0}\left(  t\right)
J_{\alpha+1}\left(  \sigma_{n}t\right)  J_{\alpha}\left(  \sigma_{n}t\right)
tdt\\
&  =\int_{  \sigma_{n}  ^{-1}}^{\pi/4}X_{0}\left(  t\right)
\left(  \frac{2}{\pi\sigma_{n}t}\right)  \left(  \sin\left(  \sigma_{n}%
t-\frac{\alpha\pi}{2}-\frac{\pi}{4}\right)  +O\left(  (\sigma
_{n}t) ^{-1}\right)  \right) \\
&  \left(  \cos\left(  \sigma_{n}t-\frac{\alpha\pi}{2}-\frac{\pi}{4}\right)
+O\left(  (\sigma_{n}t) ^{-1}\right)  \right)  tdt\\
&  =\int_{  \sigma_{n} ^{-1}}^{\pi/4}X_{0}\left(  t\right)
\left(  \frac{2}{\pi\sigma_{n}t}\right)  \left(  -\frac{1}{2}\cos\left(
2\sigma_{n}t-\alpha\pi\right)  +O\left(  \left\vert \sigma_{n}t\right\vert
^{-1}\right)  \right)  tdt.
\end{align*}
Since $X_0(t)/t\in\mathcal C^2(-\pi/2,\pi/2)$, integration by parts gives that the above integral is $O\left(
\sigma_{n}^{-2}\right)$.
Thus, by the above, by formula (5.14.5) in \cite{L}, and by the asymptotic expansions of 
$J_{\alpha}$ and $J_\alpha'$ (see \cite[page 364]{AS})
\begin{align*}
&  \int_{0}^{\pi/4}\left\vert V_{\mu_{n},\alpha}\left(  t\right)  \right\vert
^{2}dt\\
& =\frac{2^{2\alpha}\Gamma^{2}\left(  \alpha+1\right)  }{\sigma_{n}^{2\alpha
}}\frac{\pi^{2}}{32}\left(  J_{\alpha}^{\prime2}\left(  \sigma_{n}\frac{\pi
}{4}\right)  +J_{\alpha}^{2}\left(  \sigma_{n}\frac{\pi}{4}\right)  \right)
+O\left(  \frac{1}{\sigma_{n}^{2\alpha+3}}\right) \nonumber\\
&  =\frac{2^{2\alpha-2}\Gamma^{2}\left(  \alpha+1\right)  }{\sigma
_{n}^{2\alpha+1}}\left(1+\frac{2(-1)^n}{\pi\sigma_n}\sin\left((\beta-\alpha)\frac\pi 2\right)\right)+O\left(  \frac{1}{\sigma_{n}^{2\alpha+3}}\right)  ,\nonumber
\end{align*}

Similarly,%
\begin{align}
\label{second half}
&\int_{\pi/4}^{\pi/2}\left\vert W_{\mu_{n},\beta}\left(  t\right)  \right\vert
^{2}dt\\
&=\frac{2^{2\beta-2}\Gamma^{2}\left(  \beta+1\right)  }{\sigma
_{n}^{2\beta+1}}\left(1+\frac{2(-1)^n}{\pi\sigma_n}\sin\left((\alpha-\beta)\frac\pi 2\right)\right)+O\left(  \frac{1}{\sigma_{n}^{2\beta+3}}\right)  .\nonumber
\end{align}
Plugging \eqref{first half} and \eqref{second half} in \eqref{norm one} along
with Lemma \ref{ratio} allows us to deduce the expansion of $c_{n}$, and again
Lemma \ref{ratio} gives $d_{n}$.
\end{proof}

\begin{lemma}
\label{U_expansion_Bessel}Uniformly in $t\in\left(  0,\pi/4+\varepsilon
\right)  $ and in sufficiently large $n$,
\begin{align*}
u_{n}(t)&=\sqrt{2}\left(  \left(  \sigma_{n}t\right)  ^{1/2}J_{\alpha}\left(
\sigma_{n}t\right)  -\frac{1}{2}X_{0}\left(  t\right)  \left(  \sigma
_{n}t\right)  ^{1/2}\frac{J_{\alpha+1}\left(  \sigma_{n}t\right)  }{\sigma
_{n}}\right)\\
&  +O\left(  \frac{t^{2}\min\left(  1,\sigma_{n}t\right)
^{\alpha+5/2}}{\sigma_{n}^{2}}\right).
\end{align*}
When $\alpha=0$ the remainder has to be multiplied by $\log(2/\min(1,\sigma_nt))$.
Uniformly in $t\in\left(  \pi/4-\varepsilon,\pi/2\right)  $, and in sufficiently large $n$,
\begin{align*}
u_{n}(t) &  =\left(  -1\right)  ^{n}\sqrt{2}\left(  \left(  \sigma_{n}\left(
\pi/2-t\right)  \right)  ^{1/2}J_{\beta}\left(  \sigma_{n}\left(
\pi/2-t\right)  \right)  \right.  \\
&  \left.  -\frac{1}{2}X_{1}\left(  \pi/2-t\right)  \left(  \sigma_{n}\left(
\pi/2-t\right)  \right)  ^{1/2}\frac{J_{\beta+1}\left(  \sigma_{n}\left(
\pi/2-t\right)  \right)  }{\sigma_{n}}\right)  \\
&  +O\left(  \frac{\left(  \pi/2-t\right)  ^{2}\min\left(  1,\sigma_{n}\left(
\pi/2-t\right)  \right)  ^{\beta+5/2}}{\sigma_{n}^{2}}\right)  .
\end{align*}
When $\beta=0$ the remainder has to be multiplied by $\log(2/\min(1,\sigma_n(\pi/2-t))$.
In particular, the eigenfunctions $u_{n}$ are uniformly bounded on $\left(
0,\pi/2\right)  .$
\end{lemma}

\begin{proof}
The expansion in the left subinterval follows from the identity $u_{n}%
(t)=c_{n}V_{\mu_{n},\alpha}\left(  t\right)  $, along with the expansions of
$c_{n}$ in Lemma \ref{coefficients} and of $V_{\mu_{n},\alpha}\left(
t\right)  $ from Theorem \ref{BG}, and similarly for the right subinterval.
\end{proof}

We can now write the expansion for $u_{n}$ away from the endpoints in terms of sines and cosines.

\begin{lemma}
\label{U_expansion}We have%
\begin{align*}
u_{n}\left(  t\right)  &=\frac{2}{\sqrt{\pi}}\cos\left(  \sigma
_{n}t-\frac{\alpha\pi}{2}-\frac{\pi}{4}\right)\\
&  -\frac 2{\sqrt\pi}\left(    \alpha
^{2}-1/4  +tX_{0}\left(  t\right)  \right)  \frac{1}{2\sigma_{n}t}%
\sin\left(  \sigma_{n}t-\frac{\alpha\pi}{2}-\frac{\pi}{4}\right)  +O\left(
(\sigma_{n}t) ^{-2}\right)  ,
\end{align*}
uniformly for $t\in\left[  \sigma_{n}^{-1},\pi/4+\varepsilon\right)  $ and $n$
sufficiently big, and%
\begin{align*}
u_{n}\left(  t\right)   &  =\left(  -1\right)  ^{n}\frac{2}{\sqrt{\pi}}
\cos\left(  \sigma_{n}\left(  \frac{\pi}{2}-t\right)  -\frac{\beta\pi}%
{2}-\frac{\pi}{4}\right)  \\
&  -\left(  -1\right)  ^{n}\frac{2}{\sqrt{\pi}}\left(   \beta^{2}-1/4  +\left(  \frac{\pi}%
{2}-t\right)  X_{1}\left(  \frac{\pi}{2}-t\right)  \right)  \frac{1}%
{2\sigma_{n}\left(  \frac{\pi}{2}-t\right)  }\\
&\quad\times\sin\left(  \sigma_{n}\left(
\frac{\pi}{2}-t\right)  -\frac{\beta\pi}{2}-\frac{\pi}{4}\right)
  +O\left(  \left( \sigma_{n}\left(  \frac{\pi}{2}-t\right)  \right)
^{-2}\right)  ,
\end{align*}
uniformly for $t\in\left(  \pi/4-\varepsilon,\pi/2-\sigma_{n}^{-1}\right]  $
and $n$ sufficiently big.
\end{lemma}

\begin{proof}
One only has to use the expansions of $V_{\mu_{n},\alpha}\left(  t\right)  $
and $c_{n}$ in Lemma \ref{V_expansion} and Lemma \ref{coefficients}. 
%
\end{proof}

\subsection{More asymptotics}
In this section will prove second order asymptotics for the eigenvalues $\{\sigma_n\}.$ We prove it along the same lines as in Theorem $2_I$ in \cite{G2}.
\begin{lemma}
We have%
\[
X_{0}\left(  \frac{\pi}{4}\right)  +X_{1}\left(  \frac{\pi}{4}\right)
=\left(  \alpha^{2}+\beta^{2}-\frac{1}{2}\right)  \left(  \frac{\pi}{2}%
-\frac{4}{\pi}\right)  +\int_{0}^{\frac{\pi}{2}}\chi\left(  t\right)  dt.
\]

\end{lemma}

\begin{proof}
This is a simple exercise. Recall (Theorem \ref{BG}) that $X_i(t)=\int_0^t \eta_i(s)ds$, $i=0,1$.
Using the expression of $\eta_i$ it is enough to look at
\begin{align*}
&  \int_{0}^{\frac{\pi}{4}}\left(  \eta_{0}\left(  t\right)  +\eta_{1}\left(
t\right)  \right)  dt\\
  =&\int_{0}^{\frac{\pi}{4}}\left(  -\left(  \alpha^{2}+\beta^{2}-\frac{1}%
{2}\right)  \left(  \cot^{2}t-\frac{1}{t^{2}}+\tan^{2}t\right)  +\chi\left(
t\right)  +\chi\left(  \frac{\pi}{2}-t\right)  \right)  dt\\
 =&-\left(  \alpha^{2}+\beta^{2}-\frac{1}{2}\right)  \lim_{\varepsilon
\rightarrow0^{+}}\int_{\varepsilon}^{\frac{\pi}{4}}\left(  \left(  \cot
^{2}t+1\right)  -\frac{1}{t^{2}}+\left(  \tan^{2}t+1\right)  -2\right)
dt\\
&+\int_{0}^{\frac{\pi}{2}}\chi\left(  t\right)  dt\\
=&-\left(  \alpha^{2}+\beta^{2}-\frac{1}{2}\right)  \lim_{\varepsilon
\rightarrow0^{+}}\left[  -\cot t+\frac{1}{t}+\tan t-2t\right]  _{\varepsilon
}^{\frac{\pi}{4}}+\int_{0}^{\frac{\pi}{2}}\chi\left(  t\right)  dt\\
  =&\left(  \alpha^{2}+\beta^{2}-\frac{1}{2}\right)  \left(  \frac{\pi}%
{2}-\frac{4}{\pi}\right)  +\int_{0}^{\frac{\pi}{2}}\chi\left(  t\right)  dt.
\end{align*}

\end{proof}

We prove our main estimates now.
\begin{lemma}
\label{asymptotic eigenvalue two}For $n\rightarrow+\infty$%
\[
{\sigma_{n}}=2n+1+\alpha+\beta-\frac{\Theta}{4n}+O\left(  \frac{1}{n^{2}%
}\right)  ,
\]
where%
\[
\Theta=\alpha^{2}+\beta^{2}-1/2+\frac{2}{\pi}\int_{0}^{\frac{\pi}{2}}%
\chi\left(  t\right)  dt
\]

\end{lemma}

\begin{proof}
We follow the lines of the proof of Theorem $2_{I}$ in \cite{G2}. Assume
first that
\[
\sin\left(  \sigma_{n}\frac{\pi}{4}-\frac{\alpha\pi}{2}-\frac{\pi}{4}\right)
\text{ and }\sin\left(  \sigma_{n}\frac{\pi}{4}-\frac{\beta\pi}{2}-\frac{\pi
}{4}\right)
\]
are both far from zero, that is $2n\pm\left(  \alpha-\beta\right)  $ is far
from a multiple of $4$. Then replace $t$ with $\pi/4$ in both expansions of
$u_{n}$ in the above Lemma \ref{U_expansion}.
\begin{align}
\frac{\sqrt \pi}{2}u_n\left(\frac{\pi}{4}\right)
=& \cos\left(  \sigma_{n}\frac{\pi}{4}-\frac{\alpha\pi}{2}-\frac{\pi}%
{4}\right) \label{intersection}\\ 
&-\left(   \alpha^{2}-\frac14 +\frac{\pi}{4}X_{0}\left(
\frac{\pi}{4}\right)  \right)  \frac{2}{\sigma_{n}\pi}\sin\left(  \sigma
_{n}\frac{\pi}{4}-\frac{\alpha\pi}{2}-\frac{\pi}{4}\right)
\nonumber\\
&  +\sin\left(  \sigma_{n}\frac{\pi}{4}-\frac{\alpha\pi}{2}-\frac{\pi}%
{4}\right)  O\left(  \left\vert \sigma_{n}\right\vert ^{-2}\right) \nonumber\\
  =&\left(  -1\right)  ^{n} \cos\left(  \sigma_{n}\frac{\pi}{4}%
-\frac{\beta\pi}{2}-\frac{\pi}{4}\right)\nonumber\\
&  -(-1)^n\left(   \beta
^{2}-\frac14 +\frac{\pi}{4}X_{1}\left(  \frac{\pi}{4}\right)  \right)
\frac{2}{\sigma_{n}\pi}\sin\left(  \sigma_{n}\frac{\pi}{4}-\frac{\beta\pi}%
{2}-\frac{\pi}{4}\right)   \nonumber\\
&  +\left(  -1\right)  ^{n}\sin\left(  \sigma_{n}\frac{\pi}{4}-\frac{\beta\pi
}{2}-\frac{\pi}{4}\right)  O\left(  \left\vert \sigma_{n}\right\vert
^{-2}\right)  .\nonumber
\end{align}
If we set%
\begin{align*}
x  &  =\sigma_{n}\frac{\pi}{4}-\frac{\alpha\pi}{2}-\frac{\pi}{4}\\
y  &  =\sigma_{n}\frac{\pi}{4}-\frac{\beta\pi}{2}-\frac{\pi}{4}+\left(
1-\left(  -1\right)  ^{n}\right)  \frac{\pi}{2}\\
C  &  =-\left(   \alpha^{2}-\frac14  +\frac{\pi}{4}X_{0}\left(
\frac{\pi}{4}\right)  \right)  \frac{2}{\pi}+O\left(  \left\vert \sigma
_{n}\right\vert ^{-1}\right) \\
D  &  =-\left(   \beta^{2}-\frac14  +\frac{\pi}{4}X_{1}\left(
\frac{\pi}{4}\right)  \right)  \frac{2}{\pi}+O\left(  \left\vert \sigma
_{n}\right\vert ^{-1}\right) \\
\cos\gamma &  =\frac{C}{\sqrt{C^{2}+D^{2}}}\\
\sin\gamma &  =\frac{D}{\sqrt{C^{2}+D^{2}}},
\end{align*}
(if $C=D=0$ then just set $\gamma=0$).
The above identity can be written as%
\begin{align}
&\cos x+\frac{C}{\sigma_{n}}\sin x\nonumber\\
&\quad=\left(  -1\right)  ^{n}\left(  \cos\left(
y-\left(  1-\left(  -1\right)  ^{n}\right)  \frac{\pi}{2}\right)  +\frac
{D}{\sigma_{n}}\sin\left(  y-\left(  1-\left(  -1\right)  ^{n}\right)
\frac{\pi}{2}\right)  \right) \nonumber\\
&\cos x+\frac{C}{\sigma_{n}}\sin x=\cos y+\frac{D}{\sigma_{n}}\sin y\nonumber\\
&\cos x-\cos y+\frac{\sqrt{C^{2}+D^{2}}}{\sigma_{n}}\left(  \cos\gamma\sin
x-\sin\gamma\sin y\right)  =0 \label{mac}%
\end{align}
By classical trigonometric identities%
\[
\cos\gamma\sin x-\sin\gamma\sin y=\frac{\sqrt{2}}{2}\left(  F\sin\left(
\frac{x+y}{2}\right)  +E\cos\left(  \frac{x+y}{2}\right)  \right)  ,
\]
where we have set%
\begin{align*}
F  &  =\cos\left(  \frac{x-y+2\gamma+\pi/2}{2}\right)  +\sin\left(
\frac{x-y-2\gamma+\pi/2}{2}\right) \\
E  &  =\sin\left(  \frac{x-y-2\gamma+\pi/2}{2}\right)  -\cos\left(
\frac{x-y+2\gamma+\pi/2}{2}\right)  .
\end{align*}
Set%
\[
G=-2\sin\left(  \frac{x-y}{2}\right).
\]
Notice that since $2n-(\alpha-\beta)$ is far from a multiple of $4$,
then $(x-y)/2=(\beta-\alpha)\pi/4-(1-(-1)^n)\pi/4$ cannot be a multiple of $\pi$,
so that $G\neq0$.
Thus \eqref{mac} becomes%
\begin{align*}
\cos\left(  x\right)  -\cos\left(  y\right)  +\frac{\sqrt{C^{2}+D^{2}}}%
{\sigma_{n}}\frac{\sqrt{2}}{2}\left(  F\sin\left(  \frac{x+y}{2}\right)
+E\cos\left(  \frac{x+y}{2}\right)  \right)   & =0\\
\sin\left(  \frac{x+y}{2}\right)  +\frac{\sqrt{2}E\sqrt{C^{2}+D^{2}}}%
{2G\sigma_{n}+\sqrt{2}F\sqrt{C^{2}+D^{2}}}\cos\left(  \frac{x+y}{2}\right)
&  =0.
\end{align*}
Call $\tan\theta=\frac{\sqrt{2}E\sqrt{C^{2}+D^{2}}}{2G\sigma_{n}+\sqrt
{2}F\sqrt{C^{2}+D^{2}}}$, with $\theta\in(-\pi/2,\pi/2)$. Then we can rewrite the above equation as%
\begin{align*}
\frac{1}{\cos\theta}\left(  \cos\theta\sin\left(  \frac{x+y}{2}\right)
+\sin\theta\cos\left(  \frac{x+y}{2}\right)  \right)   &  =0\\
\sin\left(  \frac{x+y}{2}+\theta\right)   &  =0.
\end{align*}
This implies that for some integer $k_{n}$%
\[
\frac{x+y}{2}+\theta=k_{n}\pi,
\]
but since for $\theta\rightarrow0$,%
\[
\theta=\tan\theta+O\left(  \tan\theta^{3}\right)
\]
then%
\begin{align*}
\frac{x+y}{2}  &  =  k_{n}\pi-\tan\theta+O\left(  \tan\theta^{3}\right) \\
\sigma_{n}\frac{\pi}{4}    & =\frac{\left(  \alpha+\beta\right)  \pi}{4}%
+\frac{\pi}{4}-\left(  1-\left(  -1\right)  ^{n}\right)  \frac{\pi}{4}%
+k_{n}\pi\\
&\quad-\frac{\sqrt{2}E\sqrt{C^{2}+D^{2}}}{2G\sigma_{n}+\sqrt{2}F\sqrt
{C^{2}+D^{2}}}+O\left(  \sigma_{n}^{-3}\right) \\
\sigma_{n}  &=\alpha+\beta+1-\left(  1-\left(  -1\right)  ^{n}\right)
+4k_{n}-\frac{\sqrt{2}E\sqrt{C^{2}+D^{2}}}{\pi Gn}+O\left(  n^{-2}\right)  .
\end{align*}
Now observe that, again by trigonometric identities,
\[
\frac{\sqrt{2}E\sqrt{C^{2}+D^{2}}}{\pi G}=-\frac{C+D}{\pi}=\frac{1}{4}\left(
\alpha^{2}+\beta^{2}-1/2+\frac{2}{\pi}\int_{0}^{\frac{\pi}{2}}\chi\left(
t\right)  dt\right)  +O\left(  \sigma_{n} ^{-1}\right)
.
\]
This gives
\[
\sigma_{n}=\alpha+\beta+1-\left(  1-\left(  -1\right)  ^{n}\right)
+4k_{n}-\frac{\Theta}{4n}+O\left(  n^{-2}\right)  ,
\]
which compared with the known asymptotic $\sigma_{n}=\alpha+\beta
+1+2n+O\left(  n^{-1}\right)  $ gives%
\[
k_{n}=\left\{
\begin{array}
[c]{ll}%
n/2 & \text{if }n\text{ even}\\
\left(  n+1\right)  /2 & \text{if }n\text{ odd.}%
\end{array}
\right.
\]
Finally, if we assume that $2n\pm\left(  \alpha-\beta\right)  $ is close to a
multiple of $4$ then both%
\[
\sin\left(  \sigma_{n}\frac{\pi}{4}-\frac{\alpha\pi}{2}-\frac{\pi}{4}\right)
\text{ and }\sin\left(  \sigma_{n}\frac{\pi}{4}-\frac{\beta\pi}{2}-\frac{\pi
}{4}\right)
\]
are close to zero. Then replace $t$ with $\pi/4$ in both expansions of $u_{n}$
in Lemma \ref{U_expansion}, but this time in \eqref{intersection} we
multiply $O\left(  \left\vert \sigma_{n}\right\vert ^{-2}\right)  $ by
$\cos\left(  \sigma_{n}\frac{\pi}{4}-\frac{\alpha\pi}{2}-\frac{\pi}{4}\right)
$ and by $\cos\left(  \sigma_{n}\frac{\pi}{4}-\frac{\beta\pi}{2}-\frac{\pi}%
{4}\right)  $ rather than respectively  by $\sin\left(  \sigma_{n}\frac{\pi}{4}-\frac
{\alpha\pi}{2}-\frac{\pi}{4}\right)  $ and by $\sin\left(  \sigma_{n}\frac
{\pi}{4}-\frac{\beta\pi}{2}-\frac{\pi}{4}\right)  $. That is we write the
identity
\begin{align*}
&  \cos\left(  \sigma_{n}\frac{\pi}{4}-\frac{\alpha\pi}{2}-\frac{\pi}%
{4}\right)  -\left(  \alpha^{2}-\frac14  +\frac{\pi}{4}X_{0}\left(
\frac{\pi}{4}\right)  \right)  \frac{2}{\sigma_{n}\pi}\sin\left(  \sigma
_{n}\frac{\pi}{4}-\frac{\alpha\pi}{2}-\frac{\pi}{4}\right) \\
&  +\cos\left(  \sigma_{n}\frac{\pi}{4}-\frac{\alpha\pi}{2}-\frac{\pi}%
{4}\right)  O\left(  \sigma_{n} ^{-2}\right) \\
  =&\left(  -1\right)  ^{n}  \cos\left(  \sigma_{n}\frac{\pi}{4}%
-\frac{\beta\pi}{2}-\frac{\pi}{4}\right)\\
&  -(-1)^n\left(  \beta
^{2}-\frac14+\frac{\pi}{4}X_{1}\left(  \frac{\pi}{4}\right)  \right)
\frac{2}{\sigma_{n}\pi}\sin\left(  \sigma_{n}\frac{\pi}{4}-\frac{\beta\pi}%
{2}-\frac{\pi}{4}\right)  \\
&  +\left(  -1\right)  ^{n}\cos\left(  \sigma_{n}\frac{\pi}{4}-\frac{\beta\pi
}{2}-\frac{\pi}{4}\right)  O\left(  \sigma_{n}
^{-2}\right)  ,
\end{align*}
and conclude the proof as before.
\end{proof}

Using the second order approximations of the eigenvalues, we have the following asymptotic expansion of $\{u_n\}.$
\begin{lemma}\label{P2}
For $n$ sufficiently big,%
\begin{align*}
u_{n}\left(  t\right)  =&\frac{2}{\sqrt{\pi}}\cos\left(  \left(  2n+\nu\right)
t-\lambda\right)  \\
&-\left(  \frac{ \alpha^{2}-1/4  +tX_{0}\left(
t\right)  -\Theta t^{2}}{4\sqrt{\pi}}\right)  \frac{2}{nt}\sin\left(  \left(
2n+\nu\right)  t-\lambda\right)  +O\left(  (nt)
^{-2}\right)
\end{align*}
uniformly for $t\in\left[  n^{-1},\pi/4+\varepsilon\right)  $, where%
\begin{align*}
\nu &  =1+\alpha+\beta\\
\lambda &  =\frac{\alpha\pi}{2}+\frac{\pi}{4}.
\end{align*}
Similarly,%
\begin{align*}
u_{n}\left(  t\right)   &  =\left(  -1\right)  ^{n}\frac{2}{\sqrt{\pi}}%
\cos\left(  \left(  2n+\nu\right)  \left(  \frac{\pi}{2}-t\right)
-\lambda^{\prime}\right) \\
&  -\frac{\left(  -1\right)  ^{n}}{4\sqrt{\pi}}\left(  \beta
^{2}-1/4  +\left(  \frac{\pi}{2}-t\right)  X_{1}\left(  \frac{\pi}%
{2}-t\right)  -\Theta\left(  \frac{\pi}{2}-t\right)  ^{2}\right) \\
&\quad\times \frac
{2}{n\left(  \pi/2-t\right)  }\sin\left(  \left(  2n+\nu\right)  \left(
\frac{\pi}{2}-t\right)  -\lambda^{\prime}\right)  +O\left(  \left( n\left(  \frac{\pi}{2}-t\right)  \right)
^{-2}\right)  ,
\end{align*}
uniformly for $t\in\left(  \pi/4-\varepsilon,\pi/2-n^{-1}\right]  $, where%
\[
\lambda^{\prime}=\frac{\beta\pi}{2}+\frac{\pi}{4}.
\]

\end{lemma}

\begin{proof}
When $t\in\left[  \sigma_{n}^{-1},\pi/4+\varepsilon\right)  $ then by Lemma
\ref{U_expansion} and Lemma \ref{asymptotic eigenvalue two},
\begin{align*}
u_{n}\left(  t\right)    & =  \frac{2}{\sqrt{\pi}} \cos\left(  \sigma
_{n}t-\frac{\alpha\pi}{2}-\frac{\pi}{4}\right)\\
& \quad -\frac{2}{\sqrt{\pi}}\left(   \alpha
^{2}-1/4  +tX_{0}\left(  t\right)  \right)  \frac{1}{2\sigma_{n}t}%
\sin\left(  \sigma_{n}t-\frac{\alpha\pi}{2}-\frac{\pi}{4}\right)  +O\left(
(\sigma_{n}t) ^{-2}\right)  \\
&  =\frac{2}{\sqrt{\pi}} \cos\left(  \left(  2n+\nu-\frac{\Theta}%
{4n}+O\left(  n^{-2}\right)  \right)  t-\lambda\right)\\ 
&  -\frac{2}{\sqrt{\pi}}\frac{
\alpha^{2}-1/4  +tX_{0}\left(  t\right)  }{4}\frac{1}{nt}\sin\left(
\left(  2n+\nu-\frac{\Theta}{4n}+O\left(  n^{-2}\right)  \right)
t-\lambda\right)\\
&\quad  +O\left( ( nt) ^{-2}\right)  \\
&  =\frac{2}{\sqrt{\pi}}\cos\left(  \left(  2n+\nu\right)  t-\lambda\right)\\
&\quad-\left(  \frac{  \alpha^{2}-1/4 +tX_{0}\left(  t\right)  -\Theta
t^{2}}{4\sqrt{\pi}}\right)  \frac{2}{nt}\sin\left(  \left(  2n+\nu\right)
t-\lambda\right)  +O\left(  ( nt) ^{-2}\right)  ,
\end{align*}
and we proceed in a similar way for the interval $\left(  \pi/4-\varepsilon
,\pi/2-\sigma_{n}^{-1}\right]  .$
\end{proof}

\subsection{Differences}

Define $\Delta(u_{n})=u_{n}-u_{n+1}.$ Using the asymptotics of the eigenvalues in the last subsection we prove the following estimate for $\Delta(u_n)$:

\begin{lemma}\label{P3}
\label{DU_expansion_middle}There exist four bounded functions $Z_{j}\left(
t\right)  $, for $j=1,\ldots,4$, such that%
\begin{align*}
\Delta u_{n}\left(  t\right)  &=t\left(  Z_{1}\left(  t\right)  \cos\left(
2nt\right)  +Z_{2}\left(  t\right)  \sin\left(  2nt\right)  \right)\\&  +\frac
{1}{n}\left(  Z_{3}\left(  t\right)  \cos\left(  2nt\right)  +Z_{4}\left(
t\right)  \sin\left(  2nt\right)  \right)  +O\left(  \frac{1}{n^{2}t}\right)
\end{align*}
uniformly for $t\in\left[  n^{-1},\pi/4+\varepsilon\right)  $ and $n$
sufficiently big. A similar expansion also holds in $\left(  \pi/4-\varepsilon
,\pi/2-n^{-1}\right]  .$
\end{lemma}

\begin{proof}
Observe first that, by the asymptotic expansion for $\sigma_n$ in Lemma \ref{asymptotic eigenvalue two}, for any constant $C$, any expression of the 
form 
\begin{align*}
&\sin(\sigma_{n+1}t+C),\quad \cos(\sigma_{n+1}t+C),\\
&\sin\left(\frac{\sigma_n+\sigma_{n+1}}2t+C\right),
\quad \cos\left(\frac{\sigma_n+\sigma_{n+1}}2t+C\right)
\end{align*}
multiplied by a bounded function, and by $t$ or by $n^{-1}$, fits in the desired formula.
Plugging the expression of $u_{n}$ from Lemma \ref{U_expansion_Bessel} in
$\Delta(u_{n})$ we get
\begin{align*}
&\frac{1}{\sqrt{2}}\Delta(u_{n})(t)   =\left(  \sigma_{n}t\right)
^{1/2}J_{\alpha}\left(  \sigma_{n}t\right)  -\left(  \sigma_{n+1}t\right)
^{1/2}J_{\alpha}\left(  \sigma_{n+1}t\right) \\
&  -\frac{1}{2}X_{0}\left(  t\right)  \left(  \sigma_{n}t\right)  ^{1/2}%
\frac{J_{\alpha+1}\left(  \sigma_{n}t\right)  }{\sigma_{n}}+\frac{1}{2}%
X_{0}\left(  t\right)  \left(  \sigma_{n+1}t\right)  ^{1/2}\frac{J_{\alpha
+1}\left(  \sigma_{n+1}t\right)  }{\sigma_{n+1}}+O\left(  \frac{1}{n^{2}%
}\right)  .
\end{align*}
Now,%
\begin{align*}
&  \left(  \sigma_{n}t\right)  ^{1/2}J_{\alpha}\left(  \sigma_{n}t\right)
-\left(  \sigma_{n+1}t\right)  ^{1/2}J_{\alpha}\left(  \sigma_{n+1}t\right) \\
&  =-\left(  \sigma_{n}t\right)  ^{1/2}\int_{\sigma_{n}t}^{\sigma_{n+1}%
t}J_{\alpha}^{\prime}\left(  s\right)  ds+\frac{\left(  \sigma_{n}%
-\sigma_{n+1}\right)  t}{\left(  \sigma_{n}t\right)  ^{1/2}+\left(
\sigma_{n+1}t\right)  ^{1/2}}J_{\alpha}\left(  \sigma_{n+1}t\right)
\end{align*}
Since $\sigma_n-\sigma_{n+1}=-2+O(n^{-2})$, after doing the first order asymptotics of $J_{\alpha}%
$,  when $t\in\left[  \sigma_{n}^{-1},\pi/4+\varepsilon\right)  ,$%
\begin{align*}
&\frac{\left(  \sigma_{n}-\sigma_{n+1}\right)  t}{\left(  \sigma_{n}t\right)
^{1/2}+\left(  \sigma_{n+1}t\right)  ^{1/2}}J_{\alpha}\left(  \sigma
_{n+1}t\right)\\
& =-\sqrt{\frac{2}{\pi}}\frac 1{2n}(1+O(n^{-1}))\cos\left(
\sigma_{n+1}t-\frac{\alpha\pi}{2}-\frac{\pi}{4}\right)  +O\left( \frac{1}%
{\sigma_{n+1}^{2}t}\right),
\end{align*}
which fits in the
desired formula. On the other hand, using the second order asymptotics of
$J_{\alpha}^{\prime}$  (see \cite[page 364]{AS}) i.e.,%
\[
J_{\alpha}^{\prime}\left(  s\right)  =-\sqrt{\frac{2}{\pi s}}\left(
\sin\left(  s-\frac{\pi}{2}\alpha-\frac{\pi}{4}\right)  +\frac{4\alpha^{2}%
+3}{8s}\cos\left(  s-\frac{\pi}{2}\alpha-\frac{\pi}{4}\right)  \right)
+O(  s^{-5/2})
\]
we get%
\begin{align*}
&  \int_{\sigma_{n}t}^{\sigma_{n+1}t}J_{\alpha}^{{\prime}}\left(  s\right)
ds =-\int_{\sigma_{n}t}^{\sigma_{n+1}t}\sqrt{\frac{2}{\pi s}}\left(
\sin\left(  s-\frac{\pi}{2}\alpha-\frac{\pi}{4}\right)  \right)
ds\\
&-\int_{\sigma_{n}t}^{\sigma_{n+1}t}\sqrt{\frac{2}{\pi s}}\left(
\frac{4\alpha^{2}+3}{8s}\cos\left(  s-\frac{\pi}{2}\alpha-\frac{\pi}%
{4}\right)  \right)  ds\\
&+\left(  \sigma_{n}-\sigma_{n+1}\right)  O\left(  \sigma_{n}^{-5/2}%
t^{-3/2}\right)  .
\end{align*}
Let us put $\theta:=\frac{\pi}{2}\alpha+\frac{\pi}{4}$. Integration by parts yields%
\begin{align*}
&  -\sqrt{\frac{2}{\pi}}\int_{\sigma_{n}t}^{\sigma_{n+1}t}s^{-1/2}\sin\left(
s-\theta\right)  ds-\sqrt{\frac{2}{\pi}}\frac{4\alpha^{2}+3}{8}\int%
_{\sigma_{n}t}^{\sigma_{n+1}t}s^{-3/2}\cos\left(  s-\theta\right)  ds\\
&  =\sqrt{\frac{2}{\pi}}\left[  s^{-1/2}\cos\left(  s-\theta\right)  \right]
_{\sigma_{n}t}^{\sigma_{n+1}t}+\sqrt{\frac{2}{\pi}}  \frac{1-4\alpha
^{2}}{8} \left[  s^{-3/2}\sin\left(  s-\theta\right)  \right]
_{\sigma_{n}t}^{\sigma_{n+1}t}\\
&  +\sqrt{\frac{2}{\pi}}\left(  \frac{1-4\alpha^{2}}{8}\right)  \frac{3}%
{2}\int_{\sigma_{n}t}^{\sigma_{n+1}t}s^{-5/2}\sin\left(  s-\theta\right)  ds
\end{align*}
In order to have the desired asymptotic expansion for
\[
-\left(  \sigma_{n}t\right)  ^{1/2}\int_{\sigma_{n}t}^{\sigma_{n+1}t}%
J_{\alpha}^{{\prime}}\left(  s\right)  ds,
\]
we first look at
\[
\sqrt{\frac{2}{\pi}}\left[  s^{-1/2}\cos\left(  s-\theta\right)  \right]
_{\sigma_{n}t}^{\sigma_{n+1}t}.
\]
Adding and subtracting $-\sqrt{2/\pi}\left(  \sigma_{n}t\right)  ^{-1/2}%
\cos\left(  \sigma_{n+1}t-\theta\right)  $ in the above expression and
simplifying further, we get
\begin{align*}
&  \sqrt{\frac{2}{\pi}}\left(  \sigma_{n+1}t\right)  ^{-1/2}\cos\left(
\sigma_{n+1}t-\theta\right)  -\sqrt{\frac{2}{\pi}}(\sigma_{n}t)^{-1/2}%
\cos\left(  \sigma_{n}t-\theta\right) \\
&  =-\sqrt{\frac{2}{\pi}}\frac{2}{\left(  \sigma_{n}t\right)  ^{1/2}}%
\sin\left(  \frac{(\sigma_{n}+\sigma_{n+1})t}{2}-\theta\right)  \sin\left(
\frac{(\sigma_{n}-\sigma_{n+1})t}{2}\right) \\
& +\sqrt{\frac{2}{\pi}}\frac
{\sigma_{n+1}-\sigma_{n}}{\left(  \sigma_{n}t\right)  ^{1/2}\sigma_{n+1}%
^{1/2}}\frac{\cos\left(  \sigma_{n+1}t-\theta\right)  }{(\sigma_{n}%
^{1/2}+\sigma_{n+1}^{1/2})}.
\end{align*}
Similarly the second term $\sqrt{\frac{2}{\pi}} \frac{1-4\alpha^{2}}{8}  \left[
s^{-3/2}\sin\left(  s-\theta\right)  \right]  _{\sigma_{n}t}^{\sigma_{n+1}t}$
can be written in the form
\begin{align*}
&  \sqrt{\frac{2}{\pi}}  \frac{1-4\alpha^{2}}{8\left(  \sigma_{n}t\right)  ^{3/2}}  
{2}\cos\left(  \frac{(\sigma_{n+1}%
+\sigma_{n})t}{2}-\theta\right)  \sin\left(  \frac{(\sigma_{n}-\sigma_{n+1}%
)t}{2}\right)  +O\left(  \frac{1}{\sigma_{n}^{5/2}t^{3/2}}\right)
\end{align*}
For the third term it is easy to check that  
\[\int_{\sigma_{n}t}^{\sigma_{n+1}t}s^{-5/2}\sin\left(  s-\theta
\right)  ds=O\left(  \sigma_{n}^{-5/2}t^{-3/2}\right).
\]

Using once again   $\left(  \sigma_{n}-\sigma_{n+1}\right)
=-2+O(n^{-2})$, we deduce that
\[
2\sqrt{\frac{2}{\pi}}t^{-1}\sin\left(  \frac{(\sigma_{n}-\sigma_{n+1})}%
{2}t\right)  =-2\sqrt{\frac{2}{\pi}}t^{-1}\sin t+O\left(  n^{-2}\right)
\]
If we define the bounded function $Z(t):=-2\sqrt{\frac{2}{\pi}}t^{-1}\sin t$,
then combining the above asymptotic expansions we can finally write%
\begin{align*}
&  -\left(  \sigma_{n}t\right)  ^{1/2}\int_{\sigma_{n}t}^{\sigma_{n+1}%
t}J_{\alpha}^{\prime}\left(  s\right)  ds\\
&  =\sin\left(  \frac{(\sigma_{n}+\sigma_{n+1})t}{2}-\theta\right)
2\sqrt{\frac{2}{\pi}}\sin\left(  \frac{(\sigma_{n}-\sigma_{n+1})t}{2}\right)\\
& -\sqrt{\frac{2}{\pi}}\frac{\sigma_{n+1}-\sigma_{n}}{\sigma_{n+1}^{1/2}%
}\frac{\cos\left(  \sigma_{n+1}t-\theta\right)  }{(\sigma_{n}^{1/2}%
+\sigma_{n+1}^{1/2})}\\
&- \frac{1-4\alpha^{2}}{8}  \frac{1}%
{\sigma_{n}}\cos\left(  \frac{(\sigma_{n+1}+\sigma_{n})t}{2}-\theta\right)
2\sqrt{\frac{2}{\pi}}\frac{1}{t}\sin\left(  \frac{(\sigma_{n}-\sigma_{n+1}%
)t}{2}\right) 
  {+O\left(  \frac{1}{\sigma_{n}^{2}t}\right) }\\
&  =-t\sin\left(  \frac{(\sigma_{n}+\sigma_{n+1})t}{2}-\theta\right)
Z(t)
 +\sqrt{\frac{2}{\pi}}\frac{1}{2n}%
{\cos\left(  \sigma_{n+1}t-\theta\right)  }\\
&+ \frac{1-4\alpha^{2}}{8}  \frac{1}%
{2n}\cos\left(  \frac{(\sigma_{n+1}+\sigma_{n})t}{2}-\theta\right)
Z(t)  {+O\left(  \frac{1}{\sigma_{n}^{2}t}\right) }.
\end{align*}
This part also fits in the desired formula.

Let us handle
\[
-\frac{1}{2}X_{0}\left(  t\right)  \left(  \sigma_{n}t\right)  ^{1/2}%
\frac{J_{\alpha+1}\left(  \sigma_{n}t\right)  }{\sigma_{n}}+\frac{1}{2}%
X_{0}\left(  t\right)  \left(  \sigma_{n+1}t\right)  ^{1/2}\frac{J_{\alpha
+1}\left(  \sigma_{n+1}t\right)  }{\sigma_{n+1}}%
\]
now. We rewrite it as
\[
-\frac{1}{2}X_{0}\left(  t\right)  \left(  \left(  \sigma_{n}t\right)
^{1/2}\frac{J_{\alpha+1}\left(  \sigma_{n}t\right)  }{\sigma_{n}}-\left(
\sigma_{n+1}t\right)  ^{1/2}\frac{J_{\alpha+1}\left(  \sigma_{n+1}t\right)
}{\sigma_{n+1}}\right)  .
\]
Note that the expression in the bracket above is similar to the one dealt with
before, except for the extra factor of order $n$ in the denominator. After
adding and subtracting $\left(  \sigma_{n}t\right)  ^{1/2}\frac{J_{\alpha
+1}\left(  \sigma_{n+1}t\right)  }{\sigma_{n+1}}$ inside the bracket above we
proceed in the same way as before. Due to the extra factor of $O(n^{-1})$ we
just need to use the first order approximation of the derivative of the Bessel
function. We do not go into further details. 
\end{proof}

\begin{lemma}\label{P4}
For all  $t\in \left(  0,n^{-1}\right]  ,$%
\[
\Delta u_{n}\left(  t\right)  =-\frac{2\alpha+1}{2n}u_{n}\left(  t\right)
+O\left(  n^{-2}+t\right)  .
\]
In particular,
\[
\Delta u_{n}\left(  t\right)  =O\left(  \left(  nt\right)  ^{\alpha+\frac
{1}{2}}n^{-1}+n^{-2}+t\right),
\]
so that condition (P4) holds with $\tau=\min(1,\alpha+1/2)$.
A similar expansion holds in $\left[  \pi/2-n^{-1},\pi/2\right)  .$
\end{lemma}

\begin{proof}
By Lemma \ref{U_expansion_Bessel}, for $t\le n^{-1}$ we have%
\begin{align*}
u_{n}(t)  &  =\sqrt{2}  \left(  \sigma_{n}t\right)  ^{1/2}J_{\alpha
}\left(  \sigma_{n}t\right)  -\frac{\sqrt2}{2}X_{0}\left(  t\right)  \left(
\sigma_{n}t\right)  ^{1/2}\frac{J_{\alpha+1}\left(  \sigma_{n}t\right)
}{\sigma_{n}} \\
&\quad+O\left(  \frac{t^2(\sigma_{n}t)
^{\alpha+1/2}}{\sigma_{n}^{2}}\right) \\
 u_{n}(t) & =\sqrt{2}\left(  \sigma_{n}t\right)  ^{1/2}J_{\alpha}\left(  \sigma
_{n}t\right)  +O\left(  \left(  t\sigma_{n}\right)  ^{\alpha+5/2}{\sigma_n}^{-2}\right) .
\end{align*}
Notice that, while the first identity has to be adjusted with the usual logarithmic correction when $\alpha=0$,
the second works for $\alpha=0$ too. Thus%
\begin{align*}
\Delta u_{n}\left(  t\right)&  =\sqrt{2}\Delta\left(  \sigma_{n}^{1/2}\right)
t^{1/2}J_{\alpha}\left(  \sigma_{n}t\right)  +\sqrt{2}\left(  \sigma
_{n}t\right)  ^{1/2}\Delta\left(  J_{\alpha}\left(  \sigma_{n}t\right)
\right) \\
&\quad +O\left( (tn)^{\alpha+5/2} n^{-2}\right),
\end{align*}%
where
\[
\Delta\left(  \sigma_{n}^{1/2}\right)  =\sigma_{n}^{1/2}-\sigma_{n+1}%
^{1/2}=-\sigma_{n}^{1/2}\left(  \frac{1}{2n}+O\left(  n^{-2}\right)  \right)
.
\]
It is well known that
\[
J_{\alpha}^{\prime}\left(  s\right)  =\left\{
\begin{array}
[c]{ll}%
\dfrac{\alpha s^{\alpha-1}}{2^{\alpha}\Gamma\left(  \alpha+1\right)
}+O\left(  s^{\alpha+1}\right)  & \text{for }\alpha\neq0\\
O\left(  s\right)  & \text{for }\alpha=0,
\end{array}
\right.
\]
then, for $\alpha\neq0$%
\begin{align*}
\Delta\left(  J_{\alpha}\left( \sigma_n t\right)  \right)   &
=J_{\alpha}\left(  \sigma_{n}t\right)  -J_{\alpha}\left(  \sigma_{n+1}t\right)
  =-\int_{\sigma_{n}t}^{\sigma_{n+1}t}J_{\alpha}^{\prime}\left(  s\right)
ds\\
&  =-\frac{1}{2^{\alpha}\Gamma\left(  \alpha+1\right)  }\int_{\sigma_{n}%
t}^{\sigma_{n+1}t}\alpha s^{\alpha-1}ds+O\left(  \left(  \sigma_{n}t\right)
^{\alpha+1}t\right) \\
&  =\frac{-2\alpha\left(  \sigma_{n}t\right)  ^{\alpha}}{2^{\alpha}%
\Gamma\left(  \alpha+1\right)  2n}+O\left(  \left(  \sigma
_{n}t\right)  ^{\alpha}n^{-2}\right)  +O\left(  \left(  \sigma_{n}t\right)
^{\alpha+1}t\right) \\
&  =\frac{-2\alpha}{2n}J_{\alpha}\left(  \sigma_{n}t\right)  +O\left(
\left(  \sigma_{n}t\right)  ^{\alpha}n^{-2}\right)  +O\left(  \left(
\sigma_{n}t\right)  ^{\alpha+1}t\right)  .
\end{align*}
Finally, when $nt<1,$ after using the expansions above and rearranging terms we get%
\begin{align*} 
\Delta u_{n}\left(  t\right)   &  =\sqrt{2}\Delta\left(  \sigma_{n}%
^{1/2}\right)  t^{1/2}J_{\alpha}\left(  \sigma_{n}t\right)  +\sqrt{2}\left(
\sigma_{n}t\right)  ^{1/2}\Delta\left(  J_{\alpha}\left(  \sigma_{n}t\right)
\right) \\
&\quad +O\left( (tn)^{\alpha+5/2} {n}^{-2}\right) \\
&  =\sqrt{2}\left(  -\sigma_{n}^{1/2}\left(  \frac{1}{2n}+O\left(
n^{-2}\right)  \right)  \right)  t^{1/2}J_{\alpha}\left(  \sigma_{n}t\right)\\
&\quad  +\sqrt{2}\left(  \sigma_{n}t\right)  ^{1/2}\left(  \frac{-2\alpha}%
{2{n}}J_{\alpha}\left(  \sigma_{n}t\right)  +O\left(  \left(  \sigma
_{n}t\right)  ^{\alpha}n^{-2}\right)  +O\left(  \left(  \sigma_{n}t\right)
^{\alpha+1}t\right)  \right) \\
&\quad +O\left( (tn)^{\alpha+5/2} n^{-2}\right) \\
&  =-\frac{2\alpha+1}{2n}\sqrt{2}\left(  \sigma_{n}t\right)  ^{1/2}J_{\alpha
}\left(  \sigma_{n}t\right)  +O\left(  (nt)^{\alpha+1/2}n^{-2}+(nt)^{\alpha+5/2}t\right) \\
&  =-\frac{2\alpha+1}{2n}u_{n}\left(  t\right)  +O\left(  n^{-2}+t\right)  ,
\end{align*}
by Lemma \ref{U_expansion_Bessel}. The same formula holds for $\alpha=0.$
\end{proof}

\subsection{Pointwise convergence}
We are now ready to state the main results of this section.

\begin{theorem}\label{fits}
The eigenfunctions of $\tilde \ell$, $\{u_n\}_{n=0}^{+\infty}$, satisfy the properties (P1)-(P5).
\end{theorem}
\begin{proof}
Property (P1) is contained in Lemma \ref{U_expansion_Bessel}, (P2) in Lemma \ref{P2},
(P3) in Lemma \ref{P3} and (P4) in Lemma \ref{P4}. Property (P5) follows from the second part of
the last three mentioned lemmas.
\end{proof}

\begin{theorem}\label{equiconv_normal}
Assume that the sequences $\left\{  r_{n,N}\right\}  _{n=0}^{+\infty}$ satisfy (S1)
and (S2) and
suppose that for any compactly supported smooth function $g$ on $\left(
0,\pi/2\right)  $ and for any $t\in\left(  0,\pi/2\right)  $ one has
\begin{equation*}
\lim_{N\rightarrow+\infty}T_{N}g\left(  x\right)  - D_{N}g\left(  x\right)=0.
\end{equation*}
Then for any $f\in L^{1}\left(  \left(  0,\pi/2\right)  ,  dt\right)  $ and for any $t\in\left(  0,\pi/2\right)  $ one has%
\[
\lim_{N\rightarrow+\infty}T_{N}f\left(  t\right)
- D_{N}f \left(  t\right)
=0.
\]
Furthermore, if for each compactly supported smooth function $g$ on $\left(
0,\pi/2\right)  $
\begin{equation}
\lim_{N\rightarrow+\infty}T_{N}g\left(  x\right) \left(  t\right)  D_{N}g\left(  x\right)  =0\label{dense1}
\end{equation}
uniformly on $\left(  0,\pi/2\right)  $, then for each set $\Gamma
\subset\left(  0,\pi/2\right)  $ with positive distance from $0$ and from
$\pi/2$ and for all $f\in L^{1}\left(  \left(  0,\pi/2\right)  , dt\right)  $%
\[
\lim_{N\rightarrow+\infty}T_{N}f\left(  x\right)  -  D_{N}f\left(  x\right)  =0
\]
uniformly on $\Gamma.$
\end{theorem}

\begin{proof}
Compactly supported smooth functions are dense in $L^{1}\left(  \left(
0,\pi/2\right)  ,dt\right)$, and since the system $\left\{  u_{n}\right\}
_{n=0}^{+\infty}$ satisfies properties (P1)-(P5), we can apply Theorem \ref{equiconv_L1}
 and deduce both the pointwise and the uniform result. \end{proof}

There is a simple class of sequences $\{r_{n,N}\}$ for which one can guarantee 
condition \eqref{dense1}.

\begin{theorem} \label{equiconv_one}
Assume the sequences $\left\{  r_{n,N}\right\}  _{n=0}^{+\infty}$ satisfy (S1)
and (S2) and there exists a number $R$ such that for all $n\geq0$,%
\[
\lim_{N\rightarrow+\infty}r_{n,N}=R.
\]
Then for any $f\in L^{1}\left(  \left(  0,\pi/2\right), dt\right)  $ and for any $t\in\left(  0,\pi/2\right)  $ one has%
\[
\lim_{N\rightarrow+\infty} T_{N}f\left(  t\right)
- D_{N}f  \left(  t\right)
 =0,
\]
and the convergence is uniform on all sets $\Gamma\subset\left(
0,\pi/2\right)  $ with positive distance from $0$ and from $\pi/2$.
\end{theorem}

\begin{proof}
Let $g$ be a compactly supported smooth function on $\left(  0,\pi/2\right)
.$ For any positive integer $k$ we have%
\begin{align*}
\widehat{g}\left(  n\right) & =\int_{0}^{\pi/2}g\left(  t\right)  u_{n}\left(
t\right)  dt=\frac{1}{ \mu_{n} ^{k}}\int_{0}^{\pi/2}g\left(
t\right) \tilde \ell^{k}u_{n}\left(  t\right)  dt\\
&=\frac{1}{  \mu_{n}
^{k}}\int_{0}^{\pi/2}\tilde\ell^{k}g\left(  t\right)  u_{n}\left(  t\right)
dt=O\left(  n^{-2k}\right)
\end{align*}
and by the Lebesgue dominated convergence theorem,%
\[
\lim_{N\rightarrow+\infty}T_{N}g\left(  t\right)  =\lim_{N\rightarrow+\infty
}\sum_{n=0}^{+\infty}r_{n,N}\widehat{g}\left(  n\right)  u_{n}\left(
t\right)  =\sum_{n=0}^{+\infty}R\widehat{g}\left(  n\right)  u_{n}\left(
t\right)  =Rg\left(  t\right)  ,
\]
with uniform convergence on $\left(  0,\pi/2\right)  $, and similarly
$\lim_{N\rightarrow+\infty}D_{N}g\left(  t\right)  =Rg\left(  t\right)  $
uniformly on $\left(  0,\pi/2\right)  $.  Theorem \ref{equiconv_normal} now concludes the proof. 
\end{proof}

Thus, since for any $\theta\ge0$ the sequences $\{r_{n,N}\}$ defined by
\[
r_{n,N}=\frac{A_{N-n}^\theta}{A_N^\theta}
\]
where $A_n^\theta = {{n+\theta}\choose n}$ for $n\ge0$, and $A_n^\theta =0$ for $n<0$, satisfy the hypotheses of Theorem \ref{equiconv_one}, classical results for Ces\`aro means of Fourier series with respect to the cosine basis can be restated in exactly the same form for the Ces\`aro means (partial sums, when $\theta=0$), with respect to the basis $\{u_n\}$,
\[
\mathcal T^\theta_Nf(t)=\sum_{n=0}^{N}\frac{A_{N-n}^\theta}{A_N^\theta}\widehat f(n) u_n(t).
\]
Here is a non exhaustive list of results of this type.
\begin{itemize}
\item (M. Riesz) Let $\theta>0$, and let $f\in L^1((0,\pi/2),dt)$. Then $\mathcal T^\theta_Nf(t)\to f(t)$ at every point of continuity $t$ of $f$. The convergence is uniform on every closed set of points of continuity with positive distance from $0$ and $\pi/2$.
\item (Kahane-Katznelson) For any $E\subset (0,\pi/2)$ such that $|E|=0$ there exists a continuous function $f\in L^1((0,\pi/2),dt)$ such that $\mathcal T^0_Nf(t)$ diverges for all $t\in E$.
\item (Kolmogorov) There is a function $f\in L^1((0,\pi/2),dt)$ such that $\mathcal T^0_Nf(t)$ diverges everywhere.
\item (Carleson-Hunt) Let $p>1$ and let $f\in L^p((0,\pi/2),dt)$. Then $\mathcal T^0_Nf(t)\to f(t)$ for almost every $t\in(0,\pi/2)$. 
\end{itemize}

\section{Perturbed Jacobi operator in general form}\label{general}

In this section we consider the operator given by 
\[
Lv:=\frac{1}{A}\left(  Av^{\prime}\right)  ^{\prime}%
\] Here $A\left(  t\right)  $ is defined on $\left[  0,\pi/2\right]  $ by
\[
A\left(  t\right)  =\left(  \sin t\right)  ^{2\alpha+1}\left(  \cos t\right)
^{2\beta+1}B\left(  t\right)  ,
\]
where $\alpha,\beta>-1/2$. Symmetry with respect to $\pi/4$ shows that without loss of generality we can assume $\beta\le\alpha$. The function $B$ satisfies the following properties
\begin{enumerate} 
\item $B\in\mathcal{C}^{4}\left(  {I}\right) $, where $I$ is any open interval containing $[0,\pi/2]$ 
 \item $B\left(  t\right)  >0$
for all $t\in I.$ 
\item $B$ even with respect to $0$ and $\pi/2$ (thus,
in particular, $B^{\prime}\left(  0\right)  =B^{\prime}\left(  \pi/2\right)
=0)$.
\end{enumerate} 

The above properties imply that $B$ can be extended to a $\pi$-periodic positive function in $\mathcal C^4(\mathbb R)$, even with respect to $0$ and to $\pi/2$. In other words, we can assume without loss of generality that $I=\mathbb R$.
  When $B(t)\equiv1,$ the operator $L$
corresponds to the standard Jacobi operator.

%

The Liouville transformation $u\left(  t\right)  =A\left(  t\right)
^{1/2}v\left(  t\right)  $ gives the normal form

\[
\ell u:=-u^{\prime\prime}+\left(  \left(  \alpha^{2}-\frac{1}{4}\right)
\cot^{2}t+\left(  \beta^{2}-\frac{1}{4}\right)  \tan^{2}t-\chi\left(
t\right)  \right)  u
\]
where
\begin{align*}
\chi\left(  t\right)  &=\left(  \beta+\frac{1}{2}\right)  \frac{B^{\prime
}\left(  t\right)  }{B\left(  t\right)  }\tan t-\left(  \alpha+\frac{1}%
{2}\right)  \frac{B^{\prime}\left(  t\right)  }{B\left(  t\right)  }\cot
t+\frac{1}{4}\left(  \frac{B^{\prime}\left(  t\right)  }{B\left(  t\right)
}\right)  ^{2}-\frac{1}{2}\frac{B^{\prime\prime}\left(  t\right)  }{B\left(
t\right)  }\\
&+  2\alpha\beta+2\alpha+2\beta+\frac{3}{2}.
\end{align*}
Note that 
$\chi$ is a twice continuously differentiable function on $\mathbb R$, even with respect to $0$ and
$\pi/2$, and $\alpha\ge\beta>-1/2$, so that all the results of the previous section can be applied. 
In particular, Proposition \ref{boundary conditions} applies, so that we can
 consider the orthonormal basis $\left\{  v_{n}\left(  t\right)
\right\}  _{n=0}^{+\infty}$ on $L^{2}\left(  \left(  0,\pi/2\right)  ,A\left(
t\right)  dt\right)  $ defined by%
\[
v_{n}\left(  t\right)  :=A\left(  t\right)  ^{-1/2}u_{n}\left(  t\right)  .
\]

It may be worth emphasizing that $\{v_n\}_{n=0}^{+\infty}$ are the eigenfunctions of the self-adjoint restriction
of the operator $L$ to the space
\[
\left\{  z\in\mathcal{D}_A:\left[  z,A^{-1/2}V_{\mu,\alpha}\right]_L \left(
0\right)  =\left[  z,A^{-1/2}W_{\widetilde{\mu},\beta}\right]_L\left(
\pi/2\right)  =0\right\}  
\]
if $-1/2<\beta\leq\alpha<1$,
\[
\left\{  z\in\mathcal{D}_A:\left[  z,A^{-1/2}W_{\mu,\beta}\right]_L \left(
\pi/2\right)  =0\right\}  \]
if $-1/2<\beta<1\leq\alpha$, and 
$\mathcal{D}_A $
if $1\leq\beta\leq\alpha$,
where %
\begin{align*}
\mathcal D_A&=\{z,Az'\in AC_{\mathrm{loc}},z,Lz\in L^{2}\left(  \left(  0,\pi/2\right)  ,A\left(
t\right)  dt\right) \}\\
\left[  z,u\right]_L\left(  0\right)   &  =\lim_{t\rightarrow0+}\left(
z\left(  t\right)  \overline{A(t)u^{\prime}\left(  t\right)  }-A(t)z^{\prime}\left(
t\right)  \overline{u\left(  t\right)  }\right) \\
\left[  z,u\right]_L \left(  \pi/2\right)   &  =\lim_{t\rightarrow
\pi/2-}\left(  z\left(  t\right)  \overline{A(t)u^{\prime}\left(  t\right)
}-A(t)z^{\prime}\left(  t\right)  \overline{u\left(  t\right)  }\right).
\end{align*}
This follows easily from Lemma 3.2 and Corollary 3.1 in \cite{NZ} and Proposition \ref{boundary conditions}.

The Fourier coefficients of a function $f\in L^{2}\left(  \left(
0,\pi/2\right)  ,A\left(  t\right)  dt\right)  $ are denoted by%
\[
\mathcal{F}f\left(  n\right)  :=\int_{0}^{\pi/2}f\left(  t\right)
v_{n}\left(  t\right)  A\left(  t\right)  dt.
\]
It is immediate to observe that $f\in L^{2}\left(  \left(  0,\pi/2\right)
,A\left(  t\right)  dt\right)  $ if and only if $A^{1/2}f\in L^{2}\left(
\left(  0,\pi/2\right)  ,dt\right)  $, and in this case%
\[
\mathcal{F}f\left(  n\right)  =\widehat{\left(  A^{1/2}f\right)  }\left(
n\right)  .
\]
We want to prove an equiconvergence result for operators of the form%
\[
T_{N}^{A}f\left(  t\right)  :=\sum_{n=0}^{+\infty}r_{n,N}\mathcal{F}f\left(
n\right)  v_{n}\left(  t\right)  ,
\]
where $r_{n,N}$ satisfy properties (S1) and (S2).

\begin{theorem} \label{equiconv}
Assume the sequences $\left\{  r_{n,N}\right\}  _{n=0}^{+\infty}$ satisfy (S1)
and (S2) and there exists a number $R$ such that for all $n\geq0$,%
\[
\lim_{N\rightarrow+\infty}r_{n,N}=R.
\]
Then for any $f\in L^{1}\left(  \left(  0,\pi/2\right)  ,A^{1/2}\left(
t\right)  dt\right)  $ and for any $t\in\left(  0,\pi/2\right)  $ one has%
\[
\lim_{N\rightarrow+\infty} T_{N}^{A}f\left(  t\right)
-A^{-1/2}\left(  t\right)  D_{N}\left(  A^{1/2}f\right)  \left(  t\right)
 =0,
\]
and the convergence is uniform on all sets $\Gamma\subset\left(
0,\pi/2\right)  $ with positive distance from $0$ and from $\pi/2$.
\end{theorem}

\begin{proof}
This follows immediately from Theorem \ref{equiconv_one} and the identiity
\[
T^A_Nf(t)=A^{-1/2}(t)T_N\left(A^{1/2}f\right)(t).
\]
\end{proof}

In particular, the above theorem applies to a large class of $L^{p}\ $spaces.
Indeed, by Holder's inequality we get

\begin{proposition}\label{constraint}
Let $p>\left(  4\alpha+4\right)  /\left(  2\alpha+3\right)  $. Then
\[
L^{p}\left(  \left(  0,\pi/2\right)  ,A\left(  t\right)  dt\right)  \subset
L^{1+\varepsilon}\left(  \left(  0,\pi/2\right)  ,A^{1/2}\left(  t\right)  dt\right),
\]
for some $\varepsilon>0$.
\end{proposition}

\subsection{Pointwise convergence of partial sums}

In this section we will consider the case $r_{n,N}=1$ for $n\le N$ and $r_{n,N}=0$ for $n> N$. Thus, $T_N^A$ reduces to the partial sums operator
\[
T_{N}^{A}f\left(  t\right)  :=\sum_{n=0}^{N}\mathcal{F}f\left(
n\right)  v_{n}\left(  t\right).
\]
and similarly, 
\[
D_{N}f\left(  t\right)    =\frac{2}{\pi}\int_{0}^{\pi/2}f\left(
y\right)  dy+\frac{4}{\pi}\sum_{n=1}^{N}  \int_{0}^{\pi
/2}f\left(  y\right)  \cos\left(  2ny\right)  dy  \cos\left(
2nt\right).
\]
Our goal here is to apply the equiconvergence result in the previous section,
namely Theorem \ref{equiconv}, in order to transfer classical results about the pointwise convergence of the partial sums $D_{N}g\left(  t\right)$, to the partial sum operator $T_N^Af(t)$.

The first preliminary result is the following generalization of a well known result of C. Meaney for Jacobi polynomial expansions \cite{Meaney}.
\begin{theorem}\label{Meaney}
Let $p_0=(4\alpha+4)/(2\alpha+3)$. There exists a function 
\[f\in L^{p_0}((0,\pi/2),A(t)dt)\] 
such that $f(t)=0$ for all $t\in [\pi/4,\pi/2)$ and such that $T_N^Af(t)$ diverges a.e. on $(0,\pi/2)$. In particular, there is no $L^{p_0}$ localization.
\end{theorem}
\begin{proof} The proof goes exactly as in \cite{Meaney}. We sketch it here for sake of completeness. Observe first that 
\[
\int_0^{\pi/2}|v_n(t)|^{p'_0}A(t)dt
\ge \int_{1/n}^{\pi/4}|u_n(t)|^{p'_0}A(t)^{1-p_0/2}dt \ge c\log n
\]
by the expansion for $u_n$ given by (P2). 

Thus, the linear functional on $L^{p_0}([0,\pi/4], A(t)dt)$
\[
f \rightarrow \int_0^{\pi/4} f(t) v_n(t) A(t) dt
\]
has norm greater than $c(\log n)^{1/p'_0}$. Also, for any sequence $\varepsilon_n\to 0$, the linear functional on $L^{p_0}([0,\pi/4], A(t)dt)$
\[
f \rightarrow \int_0^{\pi/4} f(t) \frac{v_n(t)}{\varepsilon_n(\log n)^{1/p'_0}} A(t) dt
\]
has norm greater than $c\varepsilon_n^{-1}$.
By the uniform boundedness principle, there exists a function $f\in L^{p_0}([0,\pi/4], A(t)dt)$ such that 
\[
\limsup\limits_{n\to+\infty}\left|\frac{\mathcal Ff(n)}{\varepsilon_n(\log n)^{1/p'_0}}\right|=+\infty.
\]

If $T_N^Af(t)=\sum_{n=0}^N\mathcal Ff(n)v_n(t)$ converges on a set of positive measure in 
$[0,\pi/2]$, then $\mathcal Ff(n)v_n(t)\mapsto 0$ on this set, and in a subset of it with positive measure, by Egoroff's theorem, the convergence is uniform. Now observe that again by (P2), for $\varepsilon\le t\le \pi/2-\varepsilon$
\[
v_n(t)=A(t)^{-1/2}\frac 2\pi\cos((2n+\nu )t-\lambda)+O(1/n)
\]
so that by Lemma 6 in \cite{Meaney}, $\mathcal Ff(n)\mapsto 0$.
Setting $\varepsilon_n=(\log n)^{-1/(2p'_0)}$, gives the contradiction.
\end{proof}

It follows that a necessary condition for the a.e. convergence of $T_N^Af(t)$ for all $f\in L^{p}((0,\pi/2),A(t)dt)$, is that $p>p_0$. Notice that by Proposition \ref{constraint}, this is precisely the range of applicability of the equiconvergence result in Theorem \ref{equiconv}, which we may therefore use to show that in fact $p>p_0$ is also a sufficient condition.

\begin{theorem}\label{a.e.}
Let $p>p_0$ and let $f\in L^{p}((0,\pi/2),A(t)dt)$. Then $T_N^Af(t)$ converges a.e. to $f(t)$ as $N\to+\infty$.
\end{theorem}
\begin{proof}
Since $p>p_0$, by Proposition \ref{constraint}, $f\in L^{1+\varepsilon}((0,\pi/2),A^{1/2}(t)dt)$, so that by Theorem \ref{equiconv}, $T_N^Af(t)$ is equiconvergent with 
$A^{-1/2}(t)D_N(A^{1/2}f)(t)$. But again since $f\in L^{1+\varepsilon}((0,\pi/2),A^{1/2}(t)dt)$, then $A^{1/2}f\in L^{1+\varepsilon}((0,\pi/2),dt)$ so that by the Carleson-Hunt theorem $D_N(A^{1/2}f)(t)$ converges a.e. to $A^{1/2}(t)f(t)$. Thus, $A^{-1/2}(t)D_N(A^{1/2}f)(t)$ converges a.e. to $f(t)$, and so does $T_N^Af(t)$.
\end{proof}

We can also transfer to this context the classic result of Kahane-Katznelson on the divergence of partial sums of Fourier series of continuous functions.

\begin{theorem} \label{KK}
For any $E\subset (0,\pi/2)$ such that $|E|=0$ there exists a continuous function $f\in L^1((0,\pi/2),A^{1/2}(t)dt)$ such that $T_N^Af(t)$ diverges for all $t\in E$.
\end{theorem}
\begin{proof}
By the Kahane-Katznelson theorem \cite[page 67]{K}, there exists a function $F$ continuous on $[0,\pi/2]$ such that $D_NF(t)$ diverges for all $t\in E$. Thus the function $f(t):=A^{-1/2}(t)F(t)$ is continuous in $(0,\pi/2)$ and belongs to $L^1((0,\pi/2),A^{1/2}(t)dt)$. Then, by Theorem \ref{equiconv}, $T_N^Af(t)$ is equiconvergent with $A^{-1/2}(t)D_N(F)(t)$, and therefore diverges in $E$.
\end{proof}

We now want to discuss briefly the Hausdorff dimension of the sets of divergence of functions with certain $L^p$ regularity. We will focus here on the case $p=2$, leaving the discussion of the more delicate case $p\neq 2$ for future studies.

Following Stein's book \cite{Stein_sing}, we will first define the $\Lambda^{2,2}_\gamma((0,\pi/2), A(x)dx)$ spaces, i.e. certain spaces defined in terms of the $L^2$ modulus of continuity. Since we need to consider translations of functions, it is better to change our point of view a bit, and think of all our generic functions $f(x)$ as $\pi$-periodic and even functions on $\mathbb R$. Thus, the integral
\[
\int_0^{\pi/2}f(x+t) A(x)dx
\] 
will perfectly make sense, as well as the $L^2$ norm of a translated function
\[
\|f(x+t)\|_{L^{2}((0,\pi/2), A(x)dx)}=\left(\int_{0}^{\pi/2}|f(x+t)|^2A(x)dx\right)^{1/2}.
\]
\begin{definition}
For any $0<\gamma<1$ the spaces $\Lambda^{2,2}_\gamma((0,\pi/2), A(x)dx)$ consist of all functions $f$ in $L^{2}((0,\pi/2), A(x)dx)$ for which the norm
\[
\|f\|_{L^{2}((0,\pi/2), A(x)dx)}+\int_{-\pi/2}^{\pi/2}\frac{\|f(x+t)-f(x)\|^2_{L^{2}((0,\pi/2), A(x)dx)}}{|t|^{1+2\gamma}}dt
\]
is finite.
\end{definition}
\begin{theorem}\label{Hdim} Let $0<\gamma<1/2$. If $f\in \Lambda^{2,2}_\gamma((0,\pi/2), A(x)dx)$ then $T_N^Af$ diverges on a set with Hausdorff dimension less than or equal to $1-2\gamma$.
\end{theorem}
\begin{proof}
Let us fix a small $\varepsilon>0$, and define $\chi$ as a smooth, even, $\pi$-periodic function which equals $1$ in $[\varepsilon/2,\pi/2-\varepsilon/2]$ and equals $0$ in $[-\varepsilon/3,\varepsilon/3]$ and in $[\pi/2-\varepsilon/3,\pi/2+\varepsilon/3]$.

By the classic  $L^1$ localization for Fourier series, applied to $g-g\chi $ (see e.g. \cite[Vol. I, Theorem 6.2, page 52]{Z}), for all $g\in L^1((0,\pi/2),dx)$,
 $D_N(g\chi)(x)$ is uniformly equiconvergent with $D_Ng(x)$ in $[\varepsilon,\pi/2-\varepsilon]$. 
 
Since $f\in L^2((0,\pi/2),A(x)dx)$ and $2>p_0$, then Proposition \ref{constraint} and Theorem \ref{equiconv} imply that  for all $x\in[\varepsilon,\pi/2-\varepsilon]$
$T_N^Af(x)$ is uniformly equiconvergent with $A^{-1/2}(x)D_N(A^{1/2}f)(x)$, and therefore with $A^{-1/2}(x)D_N(A^{1/2}\chi f)(x)$.

Let us now show that $A^{1/2}\chi f\in \Lambda^{2,2}_\gamma([0,\pi/2],dx)$, defined as the set of all (even, $\pi$-periodic) functions $F$ in $L^{2}([0,\pi/2], dx)$ for which the norm
\[
\|F\|_{L^{2}([0,\pi/2], dx)}+\int_{-\pi/2}^{\pi/2}\frac{\|F(x+t)-F(x)\|^2_{L^{2}([0,\pi/2],dx)}}{|t|^{1+2\gamma}}dt
\]
is finite. Indeed, observe first that since $f\in L^2((0,\pi/2),A(x)dx)$, then obviously
$A^{1/2}\chi f\in L^2([0,\pi/2],dx)$ so that all we have to show is the boundedness of the integral
\[
\int_{-\varepsilon/100}^{\varepsilon/100}\frac{\|A^{1/2}\chi f(x+t)-A^{1/2}\chi f(x)\|^2_{L^{2}([0,\pi/2],dx)}}{|t|^{1+2\gamma}}dt.
\]
Set $\widetilde\chi$ as a smooth, even, $\pi$-periodic function which equals $1$ in $[\varepsilon/10,\pi/2-\varepsilon/10]$ and equals $0$ in $[-\varepsilon/20,\varepsilon/20]$ and in $[\pi/2-\varepsilon/20,\pi/2+\varepsilon/20]$. 

For all $x$ in the intervals $[0,\varepsilon/10]$ or in
$[\pi/2-\varepsilon/10, \pi/2]$ and for all $|t|<\varepsilon/100$, we have
$
A^{1/2}\chi f(x+t)-A^{1/2}\chi f(x)=0.
$
Thus,
\begin{align*}
&\int_{-\varepsilon/100}^{\varepsilon/100}\frac{\|A^{1/2}\chi f(x+t)-A^{1/2}\chi f(x)\|^2_{L^{2}([0,\pi/2],dx)}}{|t|^{1+2\gamma}}dt\\
=&\int_{-\varepsilon/100}^{\varepsilon/100}\frac{\|A^{1/2}\chi f(x+t)-A^{1/2}\chi f(x)\|^2_{L^{2}([0,\pi/2],\widetilde \chi(x)dx)}}{|t|^{1+2\gamma}}dt\\
\le & \,\,2\int_{-\varepsilon/100}^{\varepsilon/100}\frac{\|(A^{1/2}(x+t)\chi(x+t)-A^{1/2}(x)\chi(x)) f(x+t)\|^2_{L^{2}([0,\pi/2],\widetilde \chi(x)dx)}}{|t|^{1+2\gamma}}dt\\
&+2\int_{-\varepsilon/100}^{\varepsilon/100}\frac{\|A^{1/2}(x)\chi(x) (f(x+t)-f(x))\|^2_{L^{2}([0,\pi/2],\widetilde \chi(x)dx)}}{|t|^{1+2\gamma}}dt\\
\le & \,\,2\int_{-\varepsilon/100}^{\varepsilon/100}\frac{\|(A^{1/2}\chi)'\|_\infty^2|t|^2\| f(x)\|^2_{L^{2}([0,\pi/2],\widetilde \chi(x-t)dx)}}{|t|^{1+2\gamma}}dt\\
&+2\int_{-\varepsilon/100}^{\varepsilon/100}\frac{\| f(x+t)-f(x)\|^2_{L^{2}((0,\pi/2),A(x)dx)}}{|t|^{1+2\gamma}}dt\\
\end{align*}
and this is bounded by the hypotheses on $f$ and the estimate
\[
\| f(x)\|^2_{L^{2}([0,\pi/2],\widetilde \chi(x-t)dx)}\le c\| f(x)\|^2_{L^{2}((0,\pi/2),A(x)dx)}
\]
uniformly in $|t|\le \varepsilon/100$. 

It can be easily proven that $\Lambda^{2,2}_\gamma([0,\pi/2],dx)$ coincides with the potential space $\mathcal L^2_\gamma([0,\pi/2],dx)$ consisting of all even $\pi$-periodic functions $F$ such that 
\[
\sum_{n=1}^{+\infty}\left|\int_0^{\pi/2}F(x)\cos(2nx)dx\right|^2n^{2\gamma}
\]
is bounded. Finally, by \cite[Vol. II, Theorem 11.3, page 195]{Z}, it then follows that 
if $\gamma<1/2$ then $D_N(A^{1/2}\chi f)$ diverges on a set with $1-2\gamma$ outer capacity equal to $0$. It then follows that for all $\varepsilon>0$, the part of the divergence set of $T_N^Af$ contained in $[\varepsilon,\pi/2-\varepsilon]$ has $1-2\gamma$ outer capacity equal to $0$. Finally, the outer capacity of the divergence set of $T_N^Af$ in $(0,\pi/2)$ has $1-2\gamma$ outer capacity equal to $0$, and therefore Hausdorff dimension smaller than or equal to $1-2\gamma$.
\end{proof}

Finally, we will show that a slightly higher regularity than required in Theorem \ref{Hdim} gives pointwise convergence (except perhaps in $0$ and $\pi/2$).

\begin{theorem}\label{yes}
Let $1/2<\gamma<1$. If $f\in \Lambda^{2,2}_\gamma((0,\pi/2), A(x)dx)$ then $T_N^Af(x)$ converges for all $x\in(0,\pi/2)$, and the convergence is uniform away from $0$ and $\pi/2$.
\end{theorem}

\begin{proof}
Letting $\varepsilon>0$ and $\chi$ be as in the proof of the previous theorem, then $T_N^Af(x)$ is uniformly equiconvergent with $A^{-1/2}(x)D_N(A^{1/2}\chi f)(x)$ in $[\varepsilon,\pi/2-\varepsilon]$. Also, as before, $A^{1/2}\chi f\in \Lambda^{2,2}_\gamma([0,\pi/2],dx)$. By Bernstein's theorem \cite[Vol. I, Theorem 3.1, page 240, and the remark that follows]{Z}, if $\gamma>1/2$ then $D_N(A^{1/2}\chi f)(x)$ converges absolutely and uniformly in $(0,\pi/2)$. Thus, $T_N^Af(x)$ converges uniformly in $[\varepsilon,\pi/2-\varepsilon]$. 
\end{proof}
\section*{Acknowledgements}
 The second author is very grateful for the kind hospitality provided during her visit to University of Bergamo.

\end{document}